\documentclass[10pt,reqno]{amsart}
\usepackage[margin=1in]{geometry}
\usepackage{comment}
\usepackage{mathrsfs}
\usepackage{amsmath}
\usepackage{cite}
\usepackage{enumerate}
\usepackage{hyperref}
\hypersetup{
    hidelinks,
    unicode=true,
    psdextra=true,
    pdftitle={Equidistribution of expanding translates of smooth curves in homogeneous spaces under the action of a product of SO(n,1)s},
    pdfauthor={Yubin Shin},
    pdflang={en-US},
    pdfdisplaydoctitle=true
}
\usepackage{amssymb}
\usepackage{setspace}
\usepackage{mathtools}
\usepackage{needspace}
\usepackage[all,cmtip]{xy}

\newcommand{\tensor}{\otimes}

\newcommand{\R}{\mathbb{R}}
\newcommand{\C}{\mathbb{C}}
\newcommand{\Z}{\mathbb{Z}}
\newcommand{\N}{\mathbb{N}}
\newcommand{\Q}{\mathbb{Q}}

\newcommand{\M}{\mathcal{M}}

\newcommand{\SO}{\mathrm{SO}}
\newcommand{\norm}[1]{\lVert #1 \rVert}
\newcommand{\Ad}{\mathrm{Ad}}

\numberwithin{equation}{section}

\theoremstyle{plain} %% This is the default, anyway
\newtheorem{thm}{Theorem}[section]
\newtheorem{cor}[thm]{Corollary}
\newtheorem{lem}[thm]{Lemma}
\newtheorem{prop}[thm]{Proposition}

\theoremstyle{definition}
\newtheorem{defn}[thm]{Definition}

\theoremstyle{remark}
\newtheorem{rem}[thm]{Remark}

\newtheorem{notation}[thm]{Notation}

\begin{document}

%%% In the title, use a double backslash "\\" to show a linebreak:
%%% Use one of the following two forms:
%%% \title{Text of the title}
%%% or
%%% \title[Short form for the running head]{Text of the title}
\title[Equidistribution of smooth curves under $\prod \mathrm{SO}(n_i, 1)$]{Equidistribution of expanding translates of smooth curves in homogeneous spaces under the action of a product of $\mathrm{SO}(n,1)$'s }

%%% If there are multiple authors, they're described one at a time:
%%% First author: \author{} \address{} \curraddr{} \email{} \thanks{}
%%% Second author: \author{} \address{} \curraddr{} \email{} \thanks{}
%%% Third author: \author{} \address{} \curraddr{} \email{} \thanks{}
\author{Yubin Shin}

%%% In the address, show linebreaks with double backslashes:
% 연구가 수행된 곳 (Main Affiliation)
\address{Department of Mathematics, The Ohio State University, Columbus, OH 43210, USA}

% 현재 또는 곧 옮겨갈 곳 (Current Affiliation)
\curraddr{Department of Mathematics, National University of Singapore, Singapore 119076}
%%% Current address is optional.
% \curraddr{}

%%% Email address is optional.
 \email{shin.776@osu.edu}

%%% If there's a second author:
% \author{}
% \address{}
% \curraddr{}
% \email{}

%%% To have the current date inserted, use \date{\today}:
%\date{November 13th 2024}

%%% To include an abstract, uncomment the following two lines and type
%%% the abstract in between them:

 \begin{abstract} 
We study the limiting distributions of expanding translates of a compact segment of a smooth curve under a diagonal subgroup of $G=\prod_{i=1}^k\mathrm{SO}(n_i,1)$, where $G$ acts on a finite volume homogeneous space $L/\Gamma$ as a subgroup. We show that the expanding translates of the curve become equidistributed in the orbit closure of $G$, provided that Lebesgue almost every point on the curve avoids a certain countable collection of algebraic obstructions. Building on Ratner's measure classification theorem and the linearization technique, we establish the result by analyzing polynomial approximations of the curve on shrinking intervals, using Kempf's geometric invariant theory and certain estimates on maximal eigenvalues in $\mathrm{SL}(2,\mathbb{R})$-representations.
 \end{abstract}
\maketitle

%%% To include a table of contents, uncomment the following line:
% \tableofcontents

%%%-------------------------------------------------------------------
%%%-------------------------------------------------------------------
%%% Start the body of the paper here!  E.G., maybe use:
\section{Introduction}
\label{intro.ch}

\subsection{Motivation}

The equidistribution problem, initiated by Shah~\cite{Sha09SLnR}, concerns the limiting distribution of parameter measures on a curve segment that expands within a
homogeneous space under the action of a diagonal one-parameter subgroup. 
The central question in this problem is to determine precise conditions on the curve that ensure that the translated
measures do not lose mass to infinity or that they become equidistributed in the homogeneous
space. 
More precisely, the general setting of the problem can be described as follows. 
Let $G$ be a semisimple Lie group, $\Gamma$ be a lattice in $G$, and $x \in G/\Gamma$. 
Let $A = \{a(t): t \in \R\}$ be an $\R$-diagonalizable one-parameter subgroup of $G$.
The problem asks: For a curve $\psi: [0,1] \to G$, under what conditions on $\psi$ do the parameter measures supported on $a(t)\psi([0,1])x$ become equidistributed as $t \to \infty$ with respect to the unique $G$-invariant probability measure $\mu_G$ on $G/\Gamma$? This problem has interesting applications to hyperbolic geometry when $G = \mathrm{SO}(n,1)$~\cite{Sha09SOn1analytic, Sha09SOn1smooth} and to number theory when $G = \mathrm{SL}(n,\R)$~\cite{Sha09SLnR, ShahPYang24, ShahLYang}.

\subsection{Main result}
In this paper, we explore the case where $G$ is a product of groups of the form $\mathrm{SO}(n,1)$ and a given curve $\psi$ is smooth. In fact, we consider a more general setting: $G$ is a subgroup of a larger Lie group $L$ and acts on a homogeneous space $L/\Gamma$, as is done in ~\cite{LYangProduct, PYang20thesis}.

%Motivated by the above results, in this paper, we explore the conditions on a smooth curve in a homogeneous space under the action of a product of copies of $SO(n,1)$'s. 
To state our main theorem, we begin with some notation. 
Let $Q$ be a quadratic form in $n+1$ variables defined by
\begin{equation}\label{eq: quad form}
    Q(x_0,x_1,\dots,x_n)
    =2x_0x_n-(x_1^2+x_2^2+\cdots+x_{n-1}^2).
\end{equation}
We identify $\mathrm{SO}(n,1)$ with the identity component of
$\mathbf{SO}(Q)(\R)=\{g\in \mathrm{SL}(n+1,\R):
Q(gv)=Q(v),\,\forall v\in\R^{n+1}\}$.
%For $n \geq 2$, this group has two connected components.
%Throughout this paper, we let $\mathrm{SO}(n,1)$ denote its identity component $\mathrm{SO}_0(n,1)$. 

Let $G = \prod_{i =1}^k G_i$ where $G_i = \mathrm{SO}(n_i,1)$ with $n_i \geq 2$ for each $1 \leq i \leq k$. Unless otherwise specified, the index $i$ will always range from $1$ to $k$. Let $\pi_i: G \to G_i$ be the projection map onto the $i$-th factor. Let $L$ be a Lie group containing $G$, and $\Gamma$ be a lattice in $L$. 
Let $X = L/\Gamma$ and $x = l_0\cdot\Gamma \in X$ be a point whose $G$-orbit is dense in $X$. By replacing $\Gamma$ with $l_0 \Gamma l_0^{-1}$, we may assume without loss of generality that $x$ is the identity coset $x_0 = e \Gamma$.

Since $G$ is connected, we have $G\subset L^\circ$. Hence
\[
Gx_0\subset L^\circ\Gamma/\Gamma.
\]
The subset $L^\circ\Gamma/\Gamma$ is the connected component of $x_0$ in $L/\Gamma$. Since $Gx_0$ is dense in $L/\Gamma$, it follows that
\[
L=L^\circ\Gamma.
\]
Therefore, replacing $L$ by $L^\circ$ and $\Gamma$ by $\Gamma\cap L^\circ$, and retaining the same notation, we may and shall assume that $L$ is connected.

Let $A = \{a(t) = (a_i(t))_{i=1}^k \}_{t \in \R}$ be an $\R$-diagonalizable one-parameter subgroup of $G$ such that its projection $A_i := \{a_i(t)\}_{t \in \R}$ to each $G_i$ is nontrivial. By a suitable conjugation, we may write $a(t)$ as\[\left(
\begin{pmatrix} e^{\zeta_1 t} & & \\ 
& I_{n_1-1} & \\ 
& & e^{-\zeta_1 t}
\end{pmatrix}, 
\dots, 
\begin{pmatrix} e^{\zeta_k t} & & \\ 
& I_{n_k-1} & \\ 
& & e^{-\zeta_k t}
\end{pmatrix}\right)\]for some positive constants $\zeta_i>0$.
By reparametrization and reordering, we may assume that the values $\zeta_i$ are arranged in non-increasing order and the smallest value $\zeta_k$ is $1$; that is, \[\zeta_1 \geq \zeta_2 \geq \cdots \geq \zeta_k=1.\]     

Let
\begin{align*}
    P^-_i &= \{g_i \in G_i: \lim_{t \to \infty} a_i(t) g_i a_i(-t) \text{ exists in } G_i\}, \\
    P^- &= \{g \in G: \lim_{t \to \infty} a(t)ga(-t) \text{ exists in } G\}.
\end{align*}
It follows that $P^- = \prod_{i=1}^k P^-_i$. 
Then, we have
\begin{equation*}
    P^-\backslash G = \prod_{i=1}^k (P_i^- \backslash G_i)\cong \prod_{i=1}^k (\mathrm{SO}(n_i-1) \backslash \mathrm{SO}(n_i))\cong \prod_{i=1}^k\mathbb{S}^{n_i-1}.    
\end{equation*}
Let $\mathcal{I}_i: G_i \rightarrow P_i^-\backslash G_i$ and $\mathcal{I}: G \to P^-\backslash G $ be the corresponding quotient maps. We note that the action of $G_i$ on $\mathbb{S}^{n_i-1}\cong P_i^-\backslash G_i$ is via M\"obius transformations. 

Let $\mathscr{H}$ denote the collection of proper closed and connected (Lie) subgroups $H$ of $L$ such that $H\cap \Gamma$ is a lattice in $H$ and some nontrivial $\Ad_L$-unipotent one-parameter subgroup of $H$ acts ergodically on $H/(H \cap \Gamma)$ with respect to the $H$-invariant probability measure $\mu_H$.
Define $V_L = \bigoplus_{i =1}^{\dim L} (\bigwedge^i \mathfrak{L})$ where $\mathfrak{L}$ is the Lie algebra of $L$ and $L$ acts on $V_L$ via the representation induced by $\Ad_L$.
To each Lie subgroup $H$ of $L$, we associate a nonzero vector \[p_H \in \bigwedge\nolimits^{\dim H} \mathrm{Lie}(H).\]
Let 
\[
V_L^{0-}(A) = \{v \in V_L: \lim_{t \to \infty} a(t)v \text{ exists in } V_L\}.
\] 

For each $H \in \mathscr{H}$, define 
    \begin{align*}
    \Delta_{H}&=\{g\in G:\, gp_H\in V^{0-}_L(A)\}
    \end{align*}
    and define 
    \begin{equation*}
    	\begin{split}
    		N_G^1(H) = \{g \in N_G(H): \det(\Ad\, g|_{\text{Lie}(H)}) = 1  \}.
    	\end{split}
    \end{equation*}
\begin{comment}
\begin{rem}
  Each object in $\mathscr{S}$ can be thought of as a M\"obius embedding of a product of spheres into $\prod_{i=1}^k \mathbb{S}^{n_i-1}$. Thus, for any $S\in\mathscr{S}$, $(\mathcal{I}\circ u)^{-1}(S)$ is a M\"obius embedding of a product of spheres or subspaces into $\R^{\Sigma_{i=1}^k (n_i-1)}$, because each $\mathcal{I}_i\circ u_i:\R^{n_i-1} \to \mathbb{S}^{n_i-1}$ is the inverse stereographic projection.
\end{rem}
\end{comment} 

Let $I = [0,1]$ be a closed interval in $\R$. Let $\psi_i: I \rightarrow G_i$ be a curve and $\psi = (\psi_i)_{i = 1}^k : I \to G$.
For a measurable subset $B$ of $\R$, let $|B|$ denote its Lebesgue measure.

Define 
$\mathscr{H}_0$ as the subcollection of $\mathscr{H}$ consisting of $H \in \mathscr{H}$ for which $Gp_H$ is closed in $V_L$, $G$ does not fix $p_H$, $\Delta_H \neq \varnothing$ and the maximal connected $\R$-split subtorus of the center of the identity component of $N^1_G(H)$ is trivial.

\begin{thm} \label{Theorem : 1. main theorem} Let $G = \prod_{i=1}^k \mathrm{SO}(n_i,1)$, $L$, $\Gamma < L$, $X = L/\Gamma$, $x_0 \in X$ and $A = \{a(t)\}_{t \in \R} \subset G$  be defined as above. 
Let $\psi= (\psi_i)_{i=1}^k : I \to G$ be a continuous map such that $\mathcal{I}\circ \psi$ is a $C^l$-map for some integer $l \geq 2\zeta_1$ and $(\mathcal{I}_i\circ\psi_i)'(s) \neq 0$  for all $1\leq i\leq k$ and almost every $s\in I$. Suppose that 
\begin{equation}\label{condition: zero measure of spheres arising from Ratner's sets}
    |\{s \in I : \psi(s) \in \Delta_H\}| = 0 \text{ for all } H\in\mathscr{H}_0
\end{equation}
Then, for every $f \in C_c(X)$, we have 
\begin{equation*}
    \lim_{t\rightarrow \infty} \int_0^1 f(a(t)\psi(s)x_0)ds = \int_{X} f\, d\mu_L,
\end{equation*}
where $\mu_L$ is the unique $L$-invariant probability measure on $X$.
\end{thm}

Later, Proposition~\ref{prop:closed-orbit-obstruction} shows that for each $H \in \mathscr{H}_0$, there exist a proper semisimple subgroup $F$ of $G$ and an element $\xi$ of $G$ such that $A \subset F \subset \mathrm{Stab}_G(\xi p_H)$ and $\Delta_H = P^- F \xi$.
Each simple factor of $F$ is  isomorphic to $\mathrm{SO}(m,1)$ for some $m\geq 2$. Moreover, each simple factor can project nontrivially only to coordinate factors $G_i$ having the same expansion rate $\zeta_i$.
Consequently, if the expansion rates $\zeta_1,\ldots,\zeta_k$ are pairwise distinct, then $F$ can be chosen in the form
\[
F=\prod_{i=1}^k\mathrm{SO}(m_i,1),
\]
where $2\leq m_i\leq n_i$ and each $\mathrm{SO}(m_i,1)$ is viewed as a subgroup of $G_i=\mathrm{SO}(n_i,1)$ via the standard inclusion. Since $F$ is proper, $m_i<n_i$ for at least one $i$.
In particular, if $n_i=2$ for every $i$, then $\mathscr H_0=\varnothing$, and equidistribution holds for every curve satisfying the regularity and nonvanishing derivative assumptions of Theorem~\ref{Theorem : 1. main theorem}, without any additional avoidance condition.
For general $n_i$, the additional assumptions that $L=G$ and $\Gamma$ is irreducible imply that $m_i<n_i$ for every $i$.

\begin{rem}
   The condition \eqref{condition: zero measure of spheres arising from Ratner's sets} requires the curve $\psi$ to avoid the sets $\Delta_H$ up to a set of measure zero for $H \in \mathscr{H}_0$. Unlike analytic curves, a smooth curve can intersect such an obstruction set on a set of positive measure without being entirely contained in it, making this explicit measure-theoretic restriction necessary. Since $\mathscr{H}$ is countable~\cite{Rat91},
there are at most countably many such $\Delta_H$.
 Together with the fact that these obstructions are proper algebraic subsets of $G$, this countability ensures the existence of a smooth map that avoids all the obstructions almost everywhere.
\end{rem}

\begin{comment}{To understand the geometric meaning of these obstruction sets $\Delta_H$, we first note that $V_L^{0-}(A)$ is preserved by the action of $P^-$. Thus, the set $\Delta_H$ is invariant under left multiplication by $P^-$. This means that $\Delta_H$ consists of full $P^-$-cosets, and thus it is completely determined by its projection $\mathcal{I}(\Delta_H)$ in the quotient space $P^-\backslash G$.}
\end{comment}
To see the geometric meaning of $\Delta_H$, we need the following definitions. 
\begin{defn}
    Let $\mathcal{J} \subset \{1, 2, \dots, k\}$ be a nonempty set of indices, and let $m_{\mathcal{J}} \in \N$ be such that $2 \leq m_{\mathcal{J}}\leq \min_{j \in \mathcal{J}}n_j$. 
    %We consider $\mathbb{S}^{m_{\mathcal{J}}-1}$ to be a subset of each $\mathbb{S}^{n_j-1}$ for $j \in \mathcal{J}$ via the standard embedding that sends the last $n_j-m_\mathcal{J}$ coordinates to $0$. 
    For each $j \in \mathcal{J}$, let $\iota_j: \mathbb{S}^{m_\mathcal{J}-1} \to \mathbb{S}^{n_j-1}$ be the standard inclusion $\mathbb{S}^{m_\mathcal{J}-1} \hookrightarrow \mathbb{S}^{n_j-1}$ followed by a M\"obius transformation on $\mathbb{S}^{n_j-1}$.
    We define the \emph{diagonal M\"obius embedding} \[\iota_{\mathcal{J}}: \mathbb{S}^{m_\mathcal{J}-1} \to \prod_{j \in \mathcal{J}} \mathbb{S}^{n_j-1}\] by setting $\iota_{\mathcal{J}} = (\iota_j)_{j \in \mathcal{J}}$. 
\end{defn}

\begin{defn}\label{def: 1. obstruction set} 
    Let $\mathscr{P} $ be a partition of $\{1, 2, \dots, k\}$. For each $\mathcal{J} \in \mathscr{P}$, we choose a diagonal M\"obius embedding $\iota_\mathcal{J}$ as defined above. We then define the \emph{product of M\"obius embeddings} \[\iota_{\mathscr{P}} = \prod_{\mathcal{J}\in \mathscr{P}} \iota_{\mathcal{J}}: \prod_{\mathcal{J} \in \mathscr{P}}\mathbb{S}^{m_\mathcal{J}-1} \to \prod_{i=1}^k \mathbb{S}^{n_i-1}.\]
    %We define $\mathscr{S}$ as the collection of all $\mathrm{Im}(\Delta_{\mathscr{P}})$ that can also be expressed in the form $\mathrm{Im}(\Delta_\mathscr{P}) = \mathcal{I}\circ u( S_{p_H})$ for some $H \in \mathcal{H}$ for which $Gp_H$ is closed.
\end{defn}

   By these definitions, Proposition~\ref{prop:closed-orbit-obstruction} describes $\mathcal{I}(\Delta_H)$, for each $H \in \mathscr{H}_0$, as the image $\mathrm{Im}(\iota_{\mathscr{P}})$ of a product of diagonal M\"obius embeddings.

\begin{comment}
the part of each closed-orbit obstruction for which $G$ does not fix $p_H$ that lies in the open Bruhat chart $\mathcal I(N)$. More precisely, whenever $S_H$ is nonempty, $\mathcal I(\Delta_H)\cap\mathcal I(N)$ is the intersection of
$\mathcal I(N)$ with a finite union of images of products of
diagonal M\"obius embeddings; the coordinates in
$\mathcal J_{\mathrm{tor}}$ occur as point factors. On every
nontrivial diagonal block $\mathcal J$, the weights satisfy
$\zeta_{j_1}=\zeta_{j_2}$ for all $j_1,j_2\in\mathcal J$. If there
are no noncompact simple factors, this finite union consists only of
points.
\end{comment}

    %the set $\mathcal{I}(\Delta_{H})$ is determined by a partition $\mathscr{P}$ having the property that for each $\mathcal{J} \in \mathscr{P}$, we have $\zeta_{j_1} = \zeta_{j_2}$ for all $j_1, j_2 \in \mathcal{J}$. Specifically, if $m_\mathcal{J} \geq 2$ for all $\mathcal{J} \in \mathscr{P}$, then $\mathcal{I}(\Delta_{H})$ coincides exactly with the image of a single product of M\"obius embeddings $\iota_\mathscr{P}$. However, if $m_\mathcal{J} = 1$ for some $\mathcal{J} \in \mathscr{P}$, the embedding of the single point $\mathbb{S}^0$ into $\prod_{j \in \mathcal{J}} \mathbb{S}^{n_j-1}$ can map to two distinct locations. In this case, $\mathcal{I}(\Delta_{H})$ is given by the finite union of the images of $\iota_\mathscr{P}$ over all such possible choices of embeddings, resulting in a finite union of products of spheres.

We will deduce Theorem \ref{Theorem : 1. main theorem} from the following local equidistribution result.
\begin{thm}[Equidistribution of shrinking  pieces of a curve]\label{Thm : 1. shrinking equidistribution at point}
    %If $\varphi : I \to \R^{\Sigma_{i=1}^k (n_i-1)}$ is a $C^{l}$-map for some $l >\frac{2\zeta_1} {\zeta_k}$, $\varphi_i'(s) \neq 0$ for almost every $s \in I$ for all $i$,  and $\nu(\bigcup_{\Delta \in \mathscr{S}} \{s \in I : \mathcal{I}(u(\varphi(s))) \in \Delta\}) = 0$, 
    Let the notation and conditions be as in Theorem~\ref{Theorem : 1. main theorem}.
    Then there exists a Lebesgue null set $E\subset I$ such that, for any $s_0\in I\setminus E$, any $f\in C_c(X)$, and any sequences $\{s_n\}_{n\in\N}\subset I$ and $\{t_n\}_{n\in\N}\subset\R$ satisfying $s_n\to s_0$ and $t_n\to\infty$, we have
    \begin{equation}\label{equation: 1. equidist on shrinking curve}
        \lim_{n \to \infty}\int^1_0 f\big(a(t_n)\psi(s_n + e^{-\frac{1}{2}t_n}\eta)x_0\big) \, d\eta = \int_X f d\mu_L.
    \end{equation}
\end{thm}

\begin{rem}
    In Theorem~\ref{Thm : 1. shrinking equidistribution at point}, we state the result for the specific shrinking rate $e^{-\frac{1}{2}t_n}$ as it represents the optimal rate our method allows, and it simplifies the exposition. However, the argument remains valid for any slower shrinking rate $e^{-mt_n}$ with $0 < m < 1/2$, if the smoothness condition for $\mathcal{I}\circ \psi$ is adjusted to $l\geq\frac{\zeta_1}{m}$.
\end{rem}

\Needspace{12\baselineskip}
\subsection{Reduction to the expanding horospherical subgroup}
\begin{prop}\label{prop: 1. reduction to horospherical}
    The proof of Theorem~\ref{Thm : 1. shrinking equidistribution at point} reduces to the case where the curve $\psi$ is contained in the expanding horospherical subgroup $N$ of $G$ with respect to $A$: \[N = \{ g \in G : a(-t) g a(t) \to e \text{ as } t \to \infty \}.\]
\end{prop}

Let $N^-=\{g\in G: a(t)ga(-t)\to e \text{ as } t\to\infty\}$ denote the contracting horospherical subgroup of $G$ with respect to $A$. Due to the product structure of $G$, these horospherical subgroups decompose naturally as $N = \prod_{i=1}^k N_i$ and $N^- = \prod_{i=1}^k N_i^-$, where $N_i$ and $N_i^-$ are the expanding and contracting horospherical subgroups of $G_i$ with respect to $A_i$, respectively.

We defer the proof of Proposition~\ref{prop: 1. reduction to horospherical} to the beginning of Section~\ref{Section: Shrinking}.

For each $\mathbf{x} \in \prod_{i=1}^k \R^{n_i-1}$, we write $\mathbf{x} = (\mathbf{x}_1, \mathbf{x}_2, \dots, \mathbf{x}_k)$ where $\mathbf{x}_i \in \R^{n_i-1}$.
Define an isomorphism $u_i: \R^{n_i-1} \to N_i$ by 
\begin{equation*}
    u_i(\mathbf{x}_i) = 
    \begin{pmatrix} 1 & \mathbf{x}_i^t &            \frac{\norm{\mathbf{x}_i}_2^2}{2} \\
        & I_{n_i-1} & \mathbf{x}_i \\
        & & 1    
    \end{pmatrix}
\end{equation*}
for each $\mathbf{x}_i \in \R^{n_i-1}$.
Here, $\norm{\cdot}_2$ denotes the Euclidean norm.

We then define an isomorphism $u: \prod_{i=1}^k \R^{n_i-1} \to N$ by
\[u(\mathbf{x}) = (u_1(\mathbf{x}_1), u_2(\mathbf{x}_2), \dots, u_k(\mathbf{x}_k))\] 
for each $\mathbf{x} = (\mathbf{x}_1, \mathbf{x}_2, \dots, \mathbf{x}_k) \in \prod_{i=1}^k \R^{n_i-1}$.  
Similarly, we define isomorphisms $u_i^-:\R^{n_i-1} \to N^-_i$ and $u^-:\prod_{i=1}^k \R^{n_i-1} \to N^-$.

By Proposition~\ref{prop: 1. reduction to horospherical}, it suffices to consider curves of the form $\psi = u\circ \varphi$ for some curve $\varphi = (\varphi_i)_{i=1}^k: I \to \prod_{i=1}^k \R^{n_i-1}$ where $\varphi_i: I \to \R^{n_i-1}$. For the remainder of the paper, we will work directly with $u\circ\varphi$ in place of $\psi$. We note that the argument deducing Theorem~\ref{Theorem : 1. main theorem} from Theorem~\ref{Thm : 1. shrinking equidistribution at point}, which will be detailed later, applies to general curves as well as those contained in $N$; thus, establishing the result for $u\circ\varphi$ suffices to yield Theorem~\ref{Theorem : 1. main theorem}.

\subsection{Related results}
Building on the observations made in \cite{dani1984orbits, KM98}, Kleinbock and Weiss \cite{KleinbockWeiss08} recast a problem in Diophantine approximation into the language of homogeneous dynamics.
Utilizing their framework, Shah~\cite{Sha09SLnR} introduced the study of this equidistribution problem for analytic curves when $G = \mathrm{SL}(n,\R)$ in the context of the problem of Diophantine approximation. Subsequent works for analytic curves~\cite{ShahLYang,LyangSquareMatrices} and for smooth curves~\cite{ShahPYang24} 
established analogous results in different $\mathrm{SL}(n,\R)$-settings.

In \cite{Sha09SOn1smooth, Sha09SOn1analytic}, Shah established the result for both analytic and smooth curves in the single-factor $G=\mathrm{SO}(n,1)$ setting. In particular, for smooth curves, Shah showed that the equidistribution result is achieved provided that the given curve does not map any set of positive measure into any of the countably many obstructions. Furthermore, he showed that under the quotient map $\mathcal{I}: G \to P^- \backslash G \cong \mathbb{S}^{n-1}$, each obstruction set is the image of a ``M\"obius embedding'' of $\mathbb{S}^{m-1}$ into $\mathbb{S}^{n-1}$.

For the case where $G = (\mathrm{SO}(n,1))^k$ and $\varphi$ is analytic, L. Yang~\cite{LYangProduct} defined the componentwise inversion $X^{-1} = \big(\frac{X_1}{\norm{X_1}_2^2}, \dots, \frac{X_k}{\norm{X_k}_2^2}\big)$ for $X = (X_1, \dots, X_k) \in (\mathbb{R}^{n-1} \setminus \{0\})^k$. He showed that the equidistribution holds provided that there exists $s_0 \in I$ such that $(\varphi(s)- \varphi(s_0))^{-1}$ is defined for almost every $s \in I$ and $\{(\varphi(s_1)-\varphi(s_0))^{-1}-(\varphi(s_2)-\varphi(s_0))^{-1}: s_1, s_2 \in I\}$ is not contained in any proper subspace of $(\mathbb{R}^{n-1})^k$.

Meanwhile, P.~Yang~\cite{PYang20thesis} resolved this problem for analytic curves in full generality, establishing the equidistribution result for an arbitrary semisimple algebraic group $G$. This generalization was achieved by extending the concept of constraining pencils to unstable Schubert varieties, inspired by the work of Aka et al.~\cite{Aka18}.

As formulated in Theorem~\ref{Theorem : 1. main theorem}, the goal of this paper is to extend the results of \cite{Sha09SOn1analytic, Sha09SOn1smooth, LYangProduct} to expanding translates of smooth curves in products of groups of the form $\mathrm{SO}(n,1)$, and to provide Lie-theoretic and geometric conditions ensuring equidistribution.

\subsection{Paper organization and proof outline}\label{section: paper organization}
The proof of equidistribution for expanding translates of curves on a homogeneous space typically involves three main steps. 
 First, one proves the non-divergence (i.e., no escape of mass) of the parameter measures.
 Next, it must be shown that any limiting measure is invariant under a non-trivial unipotent subgroup. 
 Finally, the linearization technique is employed to demonstrate that the parameter measures do not accumulate on lower-dimensional unipotent-invariant subvarieties immersed in the homogeneous space. Ratner's theorem then guarantees that any limiting measure is the $L$-invariant probability measure on its homogeneous space.
 
In previous works (for example, \cite{Sha09SOn1smooth, Sha09SLnR, Sha09SOn1analytic, LYangProduct}), the non-divergence theorem due to Dani, Kleinbock, and Margulis (see \cite{DM93, KM98}) has been a key tool to establish the non-divergence of the parameter measures. Applying this theorem requires all the coordinate functions of the translates of the given curve to satisfy a certain growth property called $(C, \alpha)$-goodness.
However, while analytic functions possess this property, smooth functions generally do not. 
To address this, following the approach of Shah and P. Yang \cite{ShahPYang24}, we approximate $\varphi$ around each point $s_0 \in I$ by a Taylor polynomial of order \(l\) on shrinking intervals $s_0+e^{-\frac{1}{2}t}I$, to compensate for errors that expand due to the translation by $a(t)$.
We first demonstrate the equidistribution of the parameter measures supported on these polynomial curves on the shrinking intervals through a sequence of arguments presented in Sections~\ref{Section: Shrinking}, \ref{Sec: nondiv and unip inv}, and \ref{Sec: Equidistribution}.

To apply the linearization technique, an expansion result for orbits of the translates of a curve in a finite-dimensional representation of $G$, often called the ``basic lemma'', is required; this can be found in the aforementioned papers on equidistribution (e.g., \cite{Sha09SLnR, PYang20thesis}). 
In \cite{ShahPYang24}, this result, originally for a fixed-sized curve, was extended to shrinking pieces of polynomial curves for a particular case where $G = \mathrm{SL}(n,\R)$.
 In Section~\ref{Section: Shrinking}, we prove that the expansion result for shrinking polynomial curves also holds when $G$ is a product of groups of the form $\mathrm{SO}(n,1)$. The argument of Shah and P. Yang~\cite{ShahPYang24} exploits the nontrivial interactions among certain rank-one subgroups of $\mathrm{SL}(n+1,\mathbb{R})$. In our product setting, however, diagonal elements in one factor centralize the other factors, so the weight relations used in their argument do not directly carry over. We therefore develop a representation-theoretic framework adapted to this independence. Proposition~\ref{prop: 2. existence of expansion} is its key new ingredient, providing the weight estimate needed to establish expansion for shrinking curves. The proof uses the fact that, for every $g_i\in N_i$, there is a subgroup of $\mathrm{SO}(n_i,1)$ isomorphic to $\mathrm{PSL}(2,\R)$ and containing both $A_i$ and $g_i$. This allows us to use Lemma~\ref{lemma: basic} (due to Shah and L. Yang~\cite{ShahLYang}) regarding a lower bound for maximal eigenvalues in an $\mathrm{SL}(2,\R)$-representation.

In Section~\ref{Sec: nondiv and unip inv}, we follow standard arguments to establish the non-divergence and unipotent invariance of limiting measures. 

In Section~\ref{Sec: Equidistribution}, we identify obstruction sets $\Delta_H$ to equidistribution for $H \in \mathscr{H}$.
For each such $H$, when the $G$-orbit of a vector $p_H$ is not closed, we employ Kempf's invariant theory \cite{Kempf78} and a technique developed by Shah and P. Yang \cite{SY24}. 
Together, these enable us to replace the given representation and vector whose $G$-orbit is not closed with a different pair, provided that the new pair meets certain conditions. 
We explicitly construct a new, more computable pair that meets the required conditions by utilizing the standard representation $\R^{n+1}$ of $\mathrm{SO}(n,1)$.  Calculations with this new pair then show that the obstructions arising from this non-closed case are negligible, provided that the derivative of $\varphi$ in each component is nonzero almost everywhere. 
Therefore, the main obstructions originate from the case where $Gp_H$ is closed.
For the case where $G$ fixes $p_H$, the standard argument eliminates this possibility. 
We then describe $\Delta_H$ when $Gp_H$ is closed and $G$ does not fix
$p_H$. Again, using the nonvanishing derivative assumption, we show that in the closed-orbit case it suffices to consider the obstruction sets $\Delta_H$ with $H\in\mathscr H_0$. Therefore, when we view each obstruction in $P^- \backslash G$, it is the image of a product of diagonal M\"obius embeddings.
We show that avoiding them guarantees equidistribution of expanding translates of shrinking pieces of polynomial curves defined around $s_0$, for almost every $s_0 \in I$, thereby providing the proof of Theorem~\ref{Thm : 1. shrinking equidistribution at point} for the original curves on shrinking intervals, from which we deduce Theorem~\ref{Theorem : 1. main theorem} for equidistribution on the full interval \(I\).

\section{Polynomial approximation and linear dynamical results}\label{Section: Shrinking}

We begin by proving the reduction to the expanding horospherical subgroup stated in Proposition~\ref{prop: 1. reduction to horospherical}.

\begin{proof}[Proof of Proposition~\ref{prop: 1. reduction to horospherical}]
    Assume that the conclusion of Theorem~\ref{Thm : 1. shrinking equidistribution at point} holds for curves contained in $N$. 
    Let $\psi:I \to G$ be a curve satisfying the conditions required in Theorem~\ref{Thm : 1. shrinking equidistribution at point}.
    For each $i$, choose a representative $w_i \in G_i$ of the nontrivial Weyl group element such that $w_i = w_i^{-1}$ and $w_ia_i(t)w_i^{-1} = a_i(-t)$ for all $t \in \R$. Consider the Bruhat decomposition $G_i = P_i^-w_iP_i^- \cup P^-_i = P_i^-N_i \cup P_i^-w_i$. Since $P_i^-w_i$ is a single point in $P_i^-\backslash G_i$, and the first derivative of $\mathcal{I}_i \circ \psi_i$ is nonzero almost everywhere, $\{s \in I : P_i^-\psi_i(s) = P_i^- w_i \}$ is a null set, for all $i$.

    Let $\mathcal{O} = \{ s \in I : P_i^-\psi_i(s) \neq P_i^- w_i, \text{ for all } i\}$. Then, $\mathcal{O}$ is an open set and has full measure in $I$. Consequently, $\mathcal{O}$ can be written as a countable union of disjoint open intervals. Write $\mathcal{O} = \bigcup_{j = 1}^\infty I_j$. For each $s \in \mathcal{O}$, since $\psi(s) \in P^- N$, we can uniquely decompose $\psi(s) = p(s)\tilde{\psi}(s)$, where $p: \mathcal{O} \to P^-$ and $\tilde{\psi}: \mathcal{O} \to N$.

    Fix $j$ and replace the interval $I$ by $I_j$. Note that $\tilde{\psi}$ restricted to $I_j$ also satisfies all the conditions required in Theorem~\ref{Thm : 1. shrinking equidistribution at point} and is contained in $N$. Thus, the conclusion of Theorem~\ref{Thm : 1. shrinking equidistribution at point} is true for $\tilde{\psi}|_{I_j}$ by our initial assumption. That is, there is a Lebesgue null set $E_j$ in $I_j$ such that for any $s_0 \in I_j\setminus E_j$, $f \in C_c(X)$, $s_n \to s_0$, and $t_n \to \infty$, the equation~\eqref{equation: 1. equidist on shrinking curve} holds.

    Let $s_0$, $\{s_n\}_{n \in \N}$ and $\{t_n\}_{n \in \N}$ be given. For each $n$, let $\nu_{s_0, n}$ and $\lambda_{s_0, n}$ be the 
    $\eta$-parameter measures defined for each $f \in C_c(X)$ respectively by
    \[\int f \, d\nu_{s_0, n} = \int_0^1 f\big(a(t_n)\psi(s_n + e^{-\frac{1}{2}t_n}\eta)x_0\big) \, d\eta\]
    and
    \[\int f \, d\lambda_{s_0, n} = \int_0^1 f\big(a(t_n)\tilde{\psi}(s_n+e^{-\frac{1}{2}t_n}\eta)x_0\big) \, d\eta.\]
    
    Since $\lambda_{s_0, n} \to \mu_L$ as $n \to \infty$, one can show that $\nu_{s_0, n} \to \mu_L$ by using the fact that $\psi = p\cdot \tilde{\psi}$ where $p$ takes values in $P^-$, that the curve segments taken for the parameter measures $\nu_{s_0,n}$ are shrinking, that $f$ is uniformly continuous, and that $\mu_L$ is $L$-invariant.
    This implies that the conclusion of Theorem~\ref{Thm : 1. shrinking equidistribution at point} also holds for $\psi|_{I_j}$ with the same exceptional set $E_j$, thereby holding for $\psi$ on the entire interval $I$ with the exceptional set $(I\setminus \mathcal{O}) \cup \bigcup_{j=1}^\infty E_j$.
\end{proof} 

Throughout this section, let $\varphi: I \to \prod_{i=1}^k \mathbb{R}^{n_i-1}$ be a curve such that $u\circ \varphi$ satisfies the conditions stated in Theorem \ref{Theorem : 1. main theorem}.

\subsection{Approximation by Taylor polynomials}

To apply the non-divergence theorem (see \cite{DM93, KM98}),
we follow the approach of \cite{ShahPYang24} and approximate
$\varphi$ around each $s_0 \in I\setminus E$ by its Taylor polynomial.
Here $E$ is the exceptional set appearing in
Theorem~\ref{Thm : 1. shrinking equidistribution at point}.
We define $E$ below and establish its nullity in
Section~\ref{Sec: Equidistribution}.

For each $H \in \mathscr{H}$, let
\begin{align*}
    S_{H} &= \left\{\mathbf{x} \in \prod_{i=1}^k \mathbb{R}^{n_i-1}: u(\mathbf{x})p_{H} \in V^{0-}_L(A) \right\} \qquad \text{ and } \\
    I_{H} &= \{s \in I: \varphi(s) \in S_{H}\} = \{s \in I: u(\varphi(s)) \in \Delta_H\}.
\end{align*}
Let
\begin{align*}
    \mathscr{H}_{1} &= \{H \in \mathscr{H} : Gp_H \text{ is closed in } V_L \text{ and } Gp_H \neq \{p_H\}\}, \\
    \mathscr{H}_{2} &= \{H \in \mathscr{H} : Gp_H \text{ is not closed in } V_L\}.
\end{align*} We then define
\begin{equation}\label{2. E definitions}
    \begin{split}
        E_{1} &= \bigcup_{H \in \mathscr{H}_{1}} I_{H}, \\
        E_{2} &= \bigcup_{H \in \mathscr{H}_{2}} I_{H}, \\
        E_{3} &= \{0,1\}\cup \bigcup_{i = 1}^k \{s \in I : \varphi_i'(s) = 0\},
    \end{split}
\end{equation}
and let 
 \begin{equation*}
     E = E_{1}\cup E_{2}\cup E_{3}.
 \end{equation*}
The set \(E_3\) is null by the nonvanishing derivative assumption in
Theorem~\ref{Theorem : 1. main theorem}. The nullity of \(E_1\) will follow
from Corollary~\ref{cor:closed-orbit-nullity}, while that of \(E_2\) will
be established in Section~\ref{sec:nonclosed-orbits}. Together, these
results will show that \(E\) is a null set.

By the regularity assumption in Theorem~\ref{Theorem : 1. main theorem}, $\varphi$ is a $C^l$ map for some $l \geq 2\zeta_1$. Thus, Taylor's theorem implies that for $s, s+h \in I$, we have$$\varphi(s+h) = \varphi(s) + R_s(h) + o(h^l) \quad \text{as } h \to 0,$$where $R_s(h)$ is defined by
\begin{equation*} 
    R_{s}(h) = \sum_{j=1}^{l} \frac{\varphi^{(j)}(s)}{j!}h^j . 
\end{equation*}

Now, let $s_0 \in I\backslash E$. We choose a family $\{s_t\}_{t>0} \subset I$ such that $s_t \to s_0$ as $t \to \infty$. 
We first prove the following equidistribution theorem regarding the polynomial approximation $ u(R_{s_t}(\eta  e^{-t/2}))u(\varphi(s_t))$ of the shrinking curves $ \eta \mapsto u(\varphi(s_t+\eta e^{-t/2}))$.

\begin{thm}\label{thm: 2. equidistribution of polyn approx} 
   With $s_0$ and $\{s_t\}_{t>0}$ as above, we have
    \begin{equation*}
        \lim_{t \to \infty} \int_0^1 f(a(t)u(R_{s_t}(\eta\cdot e^{-\frac{1}{2}t}))u(\varphi(s_t))x_0)d\eta = \int_X f \, d\mu_L,
    \end{equation*}
    for all $f \in C_c(X)$.
\end{thm}
For each $t>0$, define the probability measure $\mu_{s_0,t}$ on $X$ by 
\begin{equation*} 
    \int_{X} f\,d\mu_{s_0,t}= \int_0^1 f(a(t)u(R_{s_t}(\eta  e^{-\frac{1}{2}t}))u(\varphi(s_t))x_0)d\eta,
\end{equation*} for all $f \in C_c(X)$. 
%We let $h = \eta\cdot e^{-mt}$. The dependence of $h$ on $\eta$ and $t$ will be omitted from the notation when no confusion can arise.  

As outlined in Section~\ref{section: paper organization}, we will prove this theorem in the remainder of this section and the subsequent sections, and then derive Theorem~\ref{Theorem : 1. main theorem} and Theorem~\ref{Thm : 1. shrinking equidistribution at point}. The proof of Theorem~\ref{thm: 2. equidistribution of polyn approx} will be completed in Section~\ref{sec:proof-polynomial-equidistribution}.

\subsection{Expansion in linear representations of a product of groups of the form \texorpdfstring{$\mathrm{SO}(n,1)$}{SO(n,1)}}

The primary goal of this subsection is to translate the representation-theoretic framework originally established for $\mathrm{SL}(n, \mathbb{R})$ in \cite[Sections 2 and 3]{ShahPYang24} to the setting of $\prod_{i=1}^k \mathrm{SO}(n_i, 1)$. This adaptation requires careful modifications, particularly to account for the structural differences arising from the product of rank-one factors.

Let $\mathfrak{g}$ and $\mathfrak{g}_i$ denote the Lie algebras of $G$ and $G_i$, respectively. Consider 
\[D_i= \begin{pmatrix}
    1 & & \\
     & 0_{n_i-1} & \\
      & & -1
\end{pmatrix} \in \mathfrak{g}_i.\] 
Let $\mathfrak{h}_i = \R D_i$ and $\mathfrak{h} = \bigoplus_{i=1}^k \mathfrak{h}_i$. Here, we implicitly identify the subgroup $G_i$ and its Lie algebra $\mathfrak{g}_i$ (and thus $D_i$ and $\mathfrak{h}_i$) with their respective images in $G$ and $\mathfrak{g}$. We define
\[H_C = \sum_{i=1}^k D_i \quad \text{and} \quad H_{\zeta} = \sum_{i=1}^k \zeta_i D_i.\] 
Note that $a(t) = \exp(t H_{\zeta})$ for all $t \in \mathbb{R}$.

Let $\sigma: G \to \mathrm{GL}(V)$ be a finite-dimensional representation of $G$ over $\R$, and $d\sigma: \mathfrak{g} \to \mathrm{End}(V)$ be the corresponding representation of the Lie algebra $\mathfrak{g}$. For simplicity, we shall write $\sigma(g)v = g\cdot v$ and $d \sigma(X)v = X\cdot v$ for any $g \in G$, $X \in \mathfrak{g}$, and $v \in V$.

Consider the weight space decomposition \[V = \bigoplus_{\lambda \in \mathfrak{h}^*} V_\lambda,\] where $V_{\lambda} = \{ v \in V: H\cdot v = \lambda(H)v, \forall H \in \mathfrak{h}\}$. For any $v \in V$, we can write $v = \sum_{\lambda \in \mathfrak{h}^*} v_\lambda$ with $v_\lambda \in V_{\lambda}$, and define
\[\Lambda_v = \{ \lambda \in \mathfrak{h}^*: v_\lambda \neq 0\}.\] 
For any subset $S \subset \mathfrak{h}^*$, we also define
\[v_S = \sum_{\lambda \in S} v_\lambda.\]

We also define $\beta_i \in \mathfrak{h}^*$ by \[\beta_i(D_j) = \delta_{ij}.\] Then, $\{\beta_i : 1 \leq i \leq k\}$ forms a basis for $\mathfrak{h}^*$.

Finally, we fix a basis of $V$ consisting of eigenvectors of $\mathfrak{h}$, and equip $V$ with the sup-norm with respect to this basis.

To show the non-divergence and equidistribution of $\{\mu_{s_0,t}\}_t$ using the Dani-Margulis non-divergence criterion and linearization technique, we need the following proposition, which corresponds to the ``basic lemmas'' established in previous works (see e.g., \cite{Sha09SLnR,ShahPYang24, PYang20thesis}).

\begin{prop}[cf.\ \cite{ShahPYang24}, Proposition 3.6]\label{prop: basic}
    Let $s_0 \in I \setminus E$ and let $\{s_t\}_{t>0} \subset I$ be a family such that $s_t \to s_0$ as $t \to \infty$. Let $\sigma: G \to \mathrm{GL}(V)$ be a finite-dimensional representation of $G$ over $\R$ as before. Then, there exist $D_2>0$ and $T>0$ such that for any $v \in V$ and $t\geq T$, we have 
    \begin{equation*}
        M_t \coloneqq \sup_{\eta \in I}\norm{a(t)u(R_{s_t}(\eta e^{-\frac{1}{2}t}))v} \geq D_2\norm{v},
    \end{equation*}
    where $\norm{\cdot}$ denotes the sup-norm on $V$.
\end{prop}

Following the observation made in \cite{ShahPYang24} (see Section 2.0.2), showing Proposition \ref{prop: basic} essentially reduces to proving Proposition~\ref{prop: 2. nonnegative integer degree} and Proposition~\ref{prop: 2. existence of expansion} below. We note that these two propositions are adapted from the relevant parts of \cite[Proposition 2.1]{ShahPYang24}, which is a central result of their work.

\begin{prop}\label{prop: 2. nonnegative integer degree}
   Let $v \in V \setminus\{0\}$ and $\mathbf{x} = (\mathbf{x}_1, \dots, \mathbf{x}_k) \in \prod_{i=1}^k \R^{n_i-1}$. Then, the following hold:
    \begin{enumerate}
        \item Given any $1 \leq i \leq k$, let $\mu \in \Lambda_{u_i(\mathbf{x}_i)v}$ (viewing $u_i(\mathbf{x}_i)$ as an element of $G$). Then, we have
        \[\mu = \lambda + m \beta_i\] 
        for some $\lambda \in \Lambda_v$ and $m \in \Z_{\geq 0}.$
        \item Let $v$ be such that $H_C\cdot v = bv$ for some $b \in \R$. Then, for all $\mu \in \Lambda_{u(\mathbf{x})v}$, we have
        \[\mu(H_C)-b \in \Z_{\geq 0}.\]
    \end{enumerate}
\end{prop}

\begin{prop}\label{prop: 2. existence of expansion}
    Let $v \in V \setminus\{0\}$ such that $H_C\cdot v = bv$ for some $b \in \R$ and let $\mathbf{x} = (\mathbf{x}_1, \dots, \mathbf{x}_k) \in \prod_{i=1}^k (\R^{n_i-1}\setminus \{0\})$. We define 
    \[S(\mathbf{x}, v ) \coloneqq \left\{\mu \in \Lambda_{u(\mathbf{x})v} : \mu(H_\zeta)-\frac{1}{2}(\mu(H_C)-b) \geq 0\right\}.\] 
    Then, $S(\mathbf{x},v) \neq \emptyset.$
\end{prop}

\begin{defn}[Lexicographical total order]
    Let $\mu, \lambda \in \mathfrak{h}^*$ be distinct. Then, $\mu - \lambda = \sum_{i=1}^k m_i \beta_i$ for some $m_i \in \mathbb{R}$. For the largest index $j$ such that $m_j \neq 0$, if $m_j > 0$, we say $\mu > \lambda$.
\end{defn}

\begin{defn}
    Let $v \in V \setminus \{0\}$ and $H \in \mathfrak{h}$. We define 
    \[\lambda^{H \text{-} \max}(v) = \max\{\lambda(H) : \lambda \in \Lambda_v\}\] 
    and 
    \[\lambda^{H \text{-} \min}(v) = \min\{\lambda(H) : \lambda \in \Lambda_v\}.\]
\end{defn}

Finally, for vectors in $\R^{n_i-1}$ parametrizing the unipotent elements in $N_i$, we define a geometric inversion.

\begin{defn} Let $d \ge 1$ be an integer.  For any vector $X \in \R^d \setminus \{0\}$, we define its inverse $X^{-1}$ by\[X^{-1} = \frac{X}{\norm{X}_2^2},\]where $\norm{\cdot}_2$ denotes the standard Euclidean norm.\end{defn}

\begin{proof}[Proof of Proposition~\ref{prop: 2. nonnegative integer degree}]
    
    \textbf{Proof of (1):} Fix $i$. If $\mathbf{x}_i =0$, then $\Lambda_{u_i(\mathbf{x}_i)v} = \Lambda_v$. Letting $\lambda = \mu$ and $m=0$, the claim holds. 
    
    Now suppose that $\mathbf{x}_i \neq 0$. Then 
\begin{equation*}
    \mathcal{X}=\begin{pmatrix}
    0 &\mathbf{x}_i^T & 0 \\
    & 0_{n_i-1} & \mathbf{x}_i \\
     & & 0
\end{pmatrix}, 
\mathcal{Y}=\begin{pmatrix}
0 & & \\
2\mathbf{x}_i^{-1} & 0_{n_i-1} & \\
0 & (2\mathbf{x}_i^{-1})^T & 0 
\end{pmatrix} \text{ and } 
 \mathcal{H}=2D_i
\end{equation*} form an $\mathfrak{sl}_2$-triple in $\mathfrak{g}_i$; that is, 
 \[
 [\mathcal{H},\mathcal{X}]=2\mathcal{X},\ 
 [\mathcal{H},\mathcal{Y}]=-2\mathcal{Y} \text{, and }
 [\mathcal{X},\mathcal{Y}]=\mathcal{H}.
 \]
 Therefore, by the standard representation theory of $\mathfrak{sl}(2,\R)$, for any $\mu \in \Lambda_{u_i(\mathbf{x}_i)v}$, we have $\mu = \lambda +m \beta_i$ for some $\lambda \in \Lambda_v$ and $m \in \Z_{\geq 0}$.

\textbf{Proof of (2):} By successively applying (1) to the sequence of vectors
    \[
        u(\mathbf{x})v = u_1(\mathbf{x}_1)\cdots u_k(\mathbf{x}_k)v, \quad u_2(\mathbf{x}_2)\cdots u_k(\mathbf{x}_k)v, \quad \dots, \quad u_k(\mathbf{x}_k)v, \quad v,
    \]
    we deduce that for any $\mu\in \Lambda_{u(\mathbf{x})v}$, there exist integers $m_i \in \Z_{\geq 0}$ and a weight $\lambda \in \Lambda_v$ such that 
    \[
        \mu = \lambda + \sum_{i=1}^k m_i\beta_i.
    \]
    By hypothesis, $H_C \cdot v = bv$, which implies $\lambda(H_C) = b$. Recall also that $\beta_i(H_C) = 1$ for all $i$. Hence, we obtain
    \begin{align*}
        \mu(H_C) - b &= \left( \lambda(H_C) + \sum_{i=1}^k m_i\beta_i(H_C) \right) - b \\
        &= \left( b + \sum_{i=1}^k m_i \right) - b \\
        &= \sum_{i=1}^k m_i.
    \end{align*}
    Since $m_i \in \Z_{\geq 0}$ for all $i$, we conclude that $\mu(H_C) - b \in \Z_{\geq 0}$, which completes the proof.
\end{proof}

 To prove Proposition \ref{prop: 2. existence of expansion}, we require the following lemma, which is adapted from Lemma 4.1 in \cite{ShahLYang}. 
 
\begin{lem}[\cite{ShahLYang}, Lemma 4.1]\label{lemma: basic} Let $V'$ be a finite-dimensional representation of $\mathrm{SL}(2,\R)$ over $\R$. Let \[u' = \begin{pmatrix}
    1 & 1\\
    0 & 1
\end{pmatrix} \in \mathrm{SL}(2,\R) \text{ and } D = \begin{pmatrix}
    1 & 0 \\
    0 & -1
\end{pmatrix} \in \mathfrak{sl}(2, \R).\] 
Decompose $V'$ into eigenspaces with respect to the action of $D$, i.e., \[V' = \bigoplus_{\lambda \in \R} V'_\lambda \text{ where } V'_\lambda = \{v \in V' : Dv = \lambda v \}.\]
For any $v \in V'\setminus\{0\}$, we write $v = \sum_{\lambda \in \R} \, v_\lambda \in V'$ where $v_\lambda$ is the $V'_\lambda$-component of $v$ and define \[\lambda^\text{max}(v) = \max \{ \lambda: v_\lambda \neq 0\}.\] % and $v^\text{max} = v_{\lambda^\text{max}(v)}$
Then, we have 
\[\lambda^{\text{max}}\big(u'\cdot v\big) \geq -\lambda^{\text{max}}(v).\]

\end{lem}

 Using this lemma, we will construct an element of $S(\mathbf{x}, v)$ to establish the proposition.

\begin{proof}[Proof of Proposition \ref{prop: 2. existence of expansion}]

Let $H_i = \sum_{j=1}^i D_j$. In particular, we have $H_k = H_C$. 
Let $r_i \in \R$ be such that 
\[ \sum_{i=1}^k r_i H_i = H_\zeta -H_C.\]
Then, $r_i = \zeta_i - \zeta_{i+1} \geq 0$ for $1 \leq i \leq k-1$. Furthermore, $r_k = 0$ since $\zeta_k = 1$. 
For any $\mu \in \mathfrak{h}^*$, observe that 
\begin{equation*}
    \begin{split}
        \mu(H_\zeta) - \frac{1}{2}(\mu(H_C)-b) & = \mu(H_\zeta - H_C) + \frac{1 }{2} (\mu(H_C)+b)\\
        & = \sum_{i=1}^k r_i\mu(H_i) + \frac{1}{2} (\mu
        (H_C) + b).
    \end{split}
\end{equation*}
From this observation, to show that $S(\mathbf{x}, v) \neq \emptyset$, it suffices to construct $\mu \in \Lambda_{u(\mathbf{x})v}$ such that 
\begin{equation}\label{eq: 2. Sxv decomposition}
    \mu(H_i)\geq 0 \text{ for all } i \text{ and } \mu(H_C) +b \geq 0.
\end{equation}  

Let \[v_i = \begin{cases}
    v & \text{ if } i = 0, \\
   u_i(\mathbf{x}_i) \cdots u_1(\mathbf{x}_1) v & \text{ if } 1 \leq i \leq k.
\end{cases}
\]
Let $L_i$ be the subgroup of $G$ generated by $u_i(\R \mathbf{x}_i) $, $\exp(\R D_i),$ and $u_i^-(\R\mathbf{x}_i^{-1})$. 
Then, $L_i$ is isomorphic to $\mathrm{PSL}(2,\R)$. We use the covering homomorphism $\mathrm{SL}(2,\R)\to L_i$ when applying Lemma~\ref{lemma: basic}.
Let $\lambda_0$ be the maximum element of $\Lambda_v$ in the lexicographical order. Note that $\lambda_0(H_C) = b$ because $v$ is an eigenvector of $H_C$ with eigenvalue $b$.
Tracing the successive action of $u_i(\mathbf{x}_i)$ starting from $v$, we inductively define $\lambda_i$ for $i =1, \dots, k$. We will then show that $\lambda_k \in S(\mathbf{x},v)$.

Consider the $L_i$-submodule \[W_i = \bigoplus_{m \in \Z} V_{\lambda_{i-1} +m \beta_i}\] and the projection map $\mathrm{Pr}_{W_i}: V \to W_i$. Let $w_i = \mathrm{Pr}_{W_i}(v_{i-1})$. Note that $w_1 = \mathrm{Pr}_{W_1}(v) \neq 0$, since $v$ has a nonzero $\lambda_0$-component. We define $m_i$ and $\lambda_i$ by \[m_i = \max\{m \in \Z: \lambda_{i-1} + m\beta_i \in \Lambda_{u_i(\mathbf{x}_i)w_i}\} \quad \text{and} \quad \lambda_i = \lambda_{i-1} + m_i \beta_i \in \Lambda_{u_i(\mathbf{x}_i)w_i}.\] 
Since $\mathrm{Pr}_{W_i}$ is an $L_i$-equivariant map, we have 
\[ 
u_i(\mathbf{x}_i)w_i = u_i(\mathbf{x}_i)\mathrm{Pr}_{W_i}(v_{i-1}) = \mathrm{Pr}_{W_i}(u_i(\mathbf{x}_i)v_{i-1}) = \mathrm{Pr}_{W_i}(v_i). 
\]
This implies $\lambda_i \in \Lambda_{v_i}$ and ensures that $w_{i+1} \neq 0$, so that $\lambda_{i+1}$ is well-defined.  As a result of this inductive construction, we obtain $\lambda_k \in \Lambda_{v_k} = \Lambda_{u(\mathbf{x})v}$.  %Furthermore, for any $i > j$, we have \begin{equation}\label{eq: 2. decrease lambda i}
   % \lambda_i(D_j) = \lambda_j(D_j) + \sum_{l=j+1}^i m_l %\beta_l(D_j) = \lambda_j(D_j).
%\end{equation}

Now, observe that
\begin{equation}\label{eq: 2. max before unip}
    \lambda^{D_i\text{-}\max}(w_i) = \lambda_0(D_i) \quad \text{ for all } i.
\end{equation}
Indeed, for any $\lambda \in \Lambda_{w_i}$, we can write 
\[
\lambda = \lambda_{i-1} + m \beta_i = \lambda_0 + \sum_{j=1}^{i-1} m_j \beta_j + m \beta_i
\]
for some $m \in \mathbb{Z}$. On the other hand, since $w_i = \mathrm{Pr}_{W_i}(v_{i-1}) = \mathrm{Pr}_{W_i}(u_{i-1}(\mathbf{x}_{i-1}) \cdots u_1(\mathbf{x}_1)v)$, there exists some $\lambda' \in \Lambda_v$ such that 
\[
\lambda' - \lambda \in \langle \beta_1, \dots, \beta_{i-1} \rangle_{\mathbb{Z}}.
\]
This implies that 
\[
\lambda' - \lambda_0 \in \langle \beta_1, \dots, \beta_{i-1} \rangle_{\mathbb{Z}} + m \beta_i.
\]
Since $\lambda_0$ is the maximum element of $\Lambda_v$ in the lexicographical order, we must have $m \leq 0$. Therefore, for any $\lambda \in \Lambda_{w_i}$, it follows that $\lambda(D_i) \leq \lambda_0(D_i)$. Furthermore, since $\lambda_{i-1} \in \Lambda_{w_i}$ and $\lambda_{i-1}(D_i) = \lambda_0(D_i)$, the maximum in \eqref{eq: 2. max before unip} is indeed achieved.

Next, for each $i$, by applying Lemma \ref{lemma: basic} to $L_i$, $u_i(\mathbf{x}_i)$, $D_i$, and $w_i$, we obtain 
\begin{equation}\label{equation: 2. D_i max basic lemma}
\begin{split}
    \lambda_i(D_i) &= \lambda^{D_i\text{-}\max}(u_i(\mathbf{x}_i)w_i) \qquad (\text{by the definition of } \lambda_i) \\
    &\geq -\lambda^{D_i\text{-}\max}(w_i) \\
    &= -\lambda_0(D_i) \qquad (\text{by \eqref{eq: 2. max before unip}}).
\end{split}
\end{equation}

Proceeding with this observation, we claim that 
\begin{equation}\label{equation: 2. eigenvalue of sum is nonnegative}
    \lambda_i(H_i) \geq 0, \qquad \text{ for all } i.
\end{equation}
Fix $i$. Consider the $L_1 \times \cdots \times L_i$-submodule $W' = \bigoplus_{(l_1, \dots, l_i) \in \mathbb{Z}^i} V_{\lambda_0 + \sum_{j=1}^i l_j \beta_j}$ and the corresponding projection map $\mathrm{Pr}_{W'}: V \to W'$. Let $w' = \mathrm{Pr}_{W'}(v)$. Then, $w'$ is an eigenvector of $H_i = \sum_{j=1}^i D_j$ with eigenvalue $\lambda_0(H_i)$. 
Note that $\lambda_i \in \Lambda_{u_i(\mathbf{x}_i) \cdots u_1(\mathbf{x}_1)w'}$. Indeed, since $\mathrm{Pr}_{W'}$ is $L_1 \times \cdots \times L_i$-equivariant, we have 
\[
u_i(\mathbf{x}_i) \cdots u_1(\mathbf{x}_1) w' = \mathrm{Pr}_{W'}(u_i(\mathbf{x}_i) \cdots u_1(\mathbf{x}_1)v) = \mathrm{Pr}_{W'}(v_i).
\]
Since $\lambda_i \in \Lambda_{v_i}$ and the weight $\lambda_i = \lambda_0 + \sum_{j=1}^i m_j \beta_j$ occurs in $W'$, we have $\lambda_i \in \Lambda_{u_i(\mathbf{x}_i) \cdots u_1(\mathbf{x}_1)w'}$.

By successively applying Proposition \ref{prop: 2. nonnegative integer degree} (1) to the sequence of vectors \[u_i(\mathbf{x}_i) \cdots u_1(\mathbf{x}_1)w', \ u_{i-1}(\mathbf{x}_{i-1}) \cdots u_1(\mathbf{x}_1)w',\ \dots \ , \ w',\] we have 
\[
\lambda_i = \lambda + \sum_{j=1}^i l_j \beta_j
\]
for some $\lambda \in \Lambda_{w'}$ and $l_1, \dots, l_i \in \mathbb{Z}_{\geq 0}$. This implies that 
\[
\begin{aligned}
\lambda_i(H_i) &= \lambda(H_i) + \sum_{j=1}^i l_j \\
&= \lambda_0(H_i) + \sum_{j=1}^i l_j \qquad (\text{because } w' \text{ is an eigenvector of } H_i) \\
&\geq \lambda_0(H_i).
\end{aligned}
\]
On the other hand, noting that $\lambda_i(D_j) = \lambda_j(D_j)$ for $i \geq j$ by construction, we have by \eqref{equation: 2. D_i max basic lemma} that 
\begin{align}\label{eq: 2. lambda i H i geq - lambda 0 H i}
    \lambda_i(H_i) = \sum_{j=1}^i\lambda_i(D_j) = \sum_{j=1}^i\lambda_j(D_j)  \geq -\sum_{j=1}^i \lambda_0(D_j) =  - \lambda_0(H_i).
\end{align}
Combining these two inequalities, we conclude that $\lambda_i(H_i) \geq 0$.

Finally, we claim that 
\begin{equation*}
    \lambda_k(H_i) \geq 0, \text{ for all } i, \text{ and } \lambda_k(H_C) +b \geq 0. 
\end{equation*}
Note that for all $i$, 
\begin{equation*}
    \lambda_k(H_i) = \lambda_k\left(\sum_{j=1}^i D_j\right) = \lambda_{k-1}\left(\sum_{j=1}^i D_j\right)
    = \cdots = \lambda_i\left(\sum_{j=1}^i D_j\right) = \lambda_i(H_i). 
\end{equation*}
Combining this with \eqref{equation: 2. eigenvalue of sum is nonnegative}, we have
$\lambda_k(H_i) \geq 0$ for all $i$.
Furthermore, we have 
\begin{align*}
    \lambda_k(H_C) + b &= \lambda_k(H_C) +\lambda_0(H_C) \qquad \qquad \qquad \text{ (since } H_C\cdot v = bv \text{ and  } \lambda_0 \in \Lambda_v)\\
    &= \lambda_k(H_k) + \lambda_0(H_k)\\
   &\geq 0 \qquad \qquad \qquad \qquad \qquad \qquad \quad \ \text{ (by \eqref{eq: 2. lambda i H i geq - lambda 0 H i})}.
\end{align*}
Therefore, from the observation made in \eqref{eq: 2. Sxv decomposition}, we conclude that $\lambda_k \in S(\mathbf{x}, v)$, which implies $S(\mathbf{x}, v) \neq \emptyset$. 
\end{proof}

To complete the proof of Proposition~\ref{prop: basic}, we also need the following definitions, Proposition~\ref{prop: 2. 0+ lower bound}, and Lemma~\ref{lemma: polynomial bound}. 

Using the notation introduced earlier in this section, let \[\Lambda = \{ \lambda \in \mathfrak{h}^*: V_{\lambda} \neq 0\}\] and \[B = \{\lambda(H_C): \lambda \in \Lambda\}.\] For any $b \in B$, we define 
\begin{equation*}
\begin{split}
    V(b) & = \{v \in V: H_C\cdot v = bv\}\\
     \Lambda(b) &= \{\lambda \in \Lambda: \lambda(H_C)=b\} \\
     \Lambda^{0+}(b) &= \left\{ \lambda \in \Lambda: \lambda(H_\zeta) - \frac{1}{2}(\lambda(H_C) -b) \geq 0 \right\} 
\end{split}
\end{equation*}
and for any $v \in V$, write 
\[v(b) = v_{\Lambda(b)} = \sum_{\lambda \in \Lambda(b)} v_\lambda.\]

\begin{prop}[Adapted from ~\cite{ShahPYang24}, Proposition 3.4]\label{prop: 2. 0+ lower bound}
    Let $\mathcal{K}$ be a compact subset of $\prod_{i=1}^k (\R^{n_i-1}\setminus \{0\})$. Then there exists $D_1>0$ such that for any $\mathbf{x} \in \mathcal{K}$, $v \in V$ and $b \in B$, we have
    \[\norm{[u(\mathbf{x})v(b)]_{\Lambda^{0+}(b)}} \geq D_1 \norm{v(b)}.\]
\end{prop}
\begin{proof}
    The argument closely follows the proof of \cite[Proposition 3.4]{ShahPYang24}, with a crucial modification relying on Proposition~\ref{prop: 2. existence of expansion}.
    
    Let $b \in B$ and define \[V^1(b) = \{ v \in V: H_C \cdot v = bv \text{ and } \norm{v} = 1\}.\] 
    Define a continuous function $f: \mathcal{K}\times V^1(b) \to \R_{\geq 0}$ by \[f(\mathbf{x}, v) = \norm{[u(\mathbf{x})v]_{\Lambda^{0+}(b)}}.\]
    
    Let $\mathbf{x} \in \mathcal{K}$ and $v \in V^1(b)$.
    By Proposition~\ref{prop: 2. existence of expansion}, we have $S(\mathbf{x},v) \neq \emptyset$. Since $S(\mathbf{x}, v) = \Lambda_{u(\mathbf{x})v} \cap  \Lambda^{0+}(b)$, it follows that $ \norm{[u(\mathbf{x})v]_{\Lambda^{0+}(b)}} \neq 0$. Hence, \[f(\mathbf{x}, v) >0.\] Since $\mathcal{K}$ and $V^1(b)$ are compact sets, we obtain \[D(b)\coloneqq \inf\{f(\mathbf{x},v): \mathbf{x} \in \mathcal{K} \text{ and } v \in V^1(b)\}>0.\]
    As $B$ is a finite set, $D_1 \coloneqq \min\{D(b): b \in B\}>0$. Then, for any $\mathbf{x} \in \mathcal{K}, v \in V$ and $b \in B$, we have $\norm{[u(\mathbf{x})v(b)]_{\Lambda^{0+}(b)}} \geq D_1 \norm{v(b)}.$
    
\end{proof}

\begin{lem}[\cite{ShahPYang24}, Lemma 3.5]\label{lemma: polynomial bound} Let $d \in \N$ and let $J \subset (0, \infty)$ be an interval of finite positive length. Then there exists a constant $C_{d,J}>0$ such that for any polynomial $f(\eta) = \sum_{j = 0}^d a_j\eta^j$ of degree at most $d$ where $a_j \in \R$,            \begin{equation*}
        \sup_{\eta \in J} |f(\eta)| \geq C_{d,J}\max_{j = 0}^d|a_j|.
    \end{equation*}
\end{lem}

The proof of Proposition~\ref{prop: basic} adapts the framework of \cite[Proposition 3.6]{ShahPYang24}. However, during this adaptation, the author noticed a technical gap in the original proof. The corrected argument presented below was developed and kindly communicated by N. Shah.

\begin{proof}[Proof of Proposition~\ref{prop: basic}]
 
Recall that $R_{s_t}(\eta e^{-\frac{1}{2}t}) = \sum_{j=1}^{l} \frac{\varphi^{(j)}(s_t)}{j!}(\eta e^{-\frac{1}{2}t})^{j}$. For $0<\eta\leq 1$, writing $h=\eta e^{-\frac{1}{2}t}$, 
we have $R_{s_t}(h) = h(\varphi'(s_t) + \epsilon(h)) = h\mathbf{y}(h)$ where 
$\epsilon(h)=\sum_{j=2}^{l} \frac{\varphi^{(j)}(s_t)}{j!}h^{j-1}$ and $\mathbf{y}(h) = \varphi'(s_t)+\epsilon(h)$.
Therefore,
\begin{equation*} \label{eq:yt}
u(R_{s_t}(h)) = u(h\mathbf{y}(h))=
\exp{((-\frac{1}{2}t + \log\eta)H_C)} \, u(\mathbf{y}(h)) \, \exp{((\frac{1}{2}t - \log\eta)H_C)}.
\end{equation*}

Let $\rho : V \to \R$ be a coordinate linear functional with respect to the basis defining the sup-norm; that is, it maps one of the basis elements
of $V$ to $\pm1$ and vanishes on the other basis elements. For any $w\in V$, observe that $\norm{w}\geq |\rho(w)|$. Let $v \in V$, $\mu \in \Lambda$, and $t>0$. 
For all $0<\eta\leq 1$, we obtain
\begin{equation}\label{equation: 2. polynomial form}
    \begin{split}
       M_t &\geq \left|\rho([a(t)u(R_{s_t}(h))v]_{\mu})\right| \\ &= |\sum_{b \in B}\rho([a(t)u(R_{s_t}(h)) v(b)]_{\mu})|\\
        & = |\sum_{b \in B} \rho([a(t)\exp((-\frac{1}{2}t + \log\eta)H_C)u(\mathbf{y}(h))\exp((\frac{1}{2}t -\log \eta)H_C)v(b)]_{\mu})|\\ 
        & = |\sum_{b \in B} \exp((\frac{1}{2}t - \log \eta)b)\rho([\exp(tH_\zeta+ (-\frac{1}{2}t + \log \eta)H_C)u(\mathbf{y}(h))v(b)]_\mu)|\\
        & = |\sum_{b \in B}\exp((\frac{1}{2}t -\log \eta)b +t\mu(H_\zeta)+(-\frac{1}{2}t + \log \eta)\mu(H_C))\rho([u(\mathbf{y}(h))v(b)]_\mu)|\\
        & = |\sum_{b \in B} \eta^{\mu(H_C) -b}\exp((\mu(H_\zeta)-\frac{1}{2}(\mu(H_C)-b))t)\rho([u(\mathbf{y}(h))v(b)]_{\mu})|.
    \end{split}
\end{equation}

Extending the representation $d\sigma: \mathfrak{g} \to \mathrm{End}(V)$ to the universal enveloping algebra $U(\mathfrak{g})$ of $\mathfrak{g}$, we have
\begin{equation}\label{equation: 2. polynomial form of y}
\begin{split}
    \sigma(u(\mathbf{y}(h))) &= \sigma(u(\varphi'(s_t))) + \sum_{j=1}^{d_1}h^j d\sigma(X_j(s_t)) \\
    & = \sigma(u(\varphi'(s_t))) + \sum_{j=1}^{d_1}(\eta e^{-\frac{1}{2}t})^j d\sigma(X_j(s_t)) \qquad \text{(because } h = \eta e^{-\frac{1}{2}t})
\end{split}  
 \end{equation} for some $d_1 \in \N$ and $X_1(s_t), \dots, X_{d_1}(s_t) \in U(\mathfrak{g})$. Thus,
 $\rho([u(\mathbf{y}(h))v(b)]_\mu)$ is a polynomial in the variable $\eta$ of degree at most $d_1$.
 Moreover, if $\rho([u(\mathbf{y}(h))v(b)]_{\mu}) \neq 0$, then $\mu \in \Lambda_{ u(\mathbf{y}(h))v(b)}$, which implies $\mu(H_C)-b \in \Z_{\geq 0}$ by Proposition~\ref{prop: 2. nonnegative integer degree} (2). Therefore, there exists $d\in\N$, independent of $t$, $v$, and $\mu$, such that the expression inside the absolute value in \eqref{equation: 2. polynomial form} is a polynomial in $\eta$ of degree at most $d$.

Since $V$ is finite-dimensional, we can write $B = \{ \lambda(H_C): \lambda \in \Lambda\} = \{ b_1 > b_2 > \cdots > b_n\}$ for some $n \in \N$. Since $\varphi$ is $C^l$ and $s_t\to s_0$, there exist $T_0>0$ and $C_1>0$ such that
\[
    \sup_{t\geq T_0}\max_{1\leq j\leq d_1}\norm{d\sigma(X_j(s_t))}_{\mathrm{op}}<C_1,
\]
where $\norm{\cdot}_{\mathrm{op}}$ denotes the operator norm on $\mathrm{End}(V)$ with respect to the sup-norm on $V$. In what follows, we write \(X_j\) for the operator \(d\sigma(X_j(s_t))\).

As $s_0 \in I\setminus E$, we have $\varphi'_i(s_0) \neq 0$ for all $i$. Since $s_t \to s_0$ as $t \to \infty$, there exists $T_1\geq T_0$ such that for any $t\geq T_1$, $ \varphi_i'(s_t)   \neq 0$ for all $i$. Note that, increasing $T_1$ if necessary, $\{\varphi'(s_t) : t \geq T_1\} \cup \{\varphi'(s_0)\}$ is contained in a compact subset of $\prod_{i=1}^k (\mathbb{R}^{n_i-1}\setminus \{0\})$. Then, by Proposition~\ref{prop: 2. 0+ lower bound}, there exists $D_1>0$ such that for any $t \geq T_1$, $v \in V$, and $b \in B$, we have
\begin{equation}\label{equation: 2. 0+bound}
    \norm{[u(\varphi'(s_t))v(b)]_{\Lambda^{0+}(b)}} \geq D_1\norm{v(b)}.
\end{equation}

We apply Lemma~\ref{lemma: polynomial bound} with $J=(0,1]$ and write $C_{d,I}=C_{d,J}$.

Let $v \in V$. Now, we claim that there exists $T>0$ such that for any $t \geq T$, we have
\begin{equation}\label{equation: induction hypo}
    M_t \geq C_{d,I}D_1(1+C_1D_1^{-1})^{-(j-1)}\norm{v(b_j)}, \qquad\text{ for all } 1 \leq j \leq n,
\end{equation} where $C_{d,I}$ is the constant given by Lemma~\ref{lemma: polynomial bound}. If $v(b_j) = 0$, the claim holds trivially. Thus, without loss of generality, we may assume $v(b_j) \neq 0$ for all $1 \leq j \leq n$.
Set $T=T_1$ and fix $t\geq T$. For each $1 \leq j \leq n$, choose $\mu_j \in \Lambda^{0+}(b_j)$ and a coordinate linear functional $\rho_j$ such that
\begin{equation*}
    \norm{[u(\varphi'(s_t))v(b_j)]_{\Lambda^{0+}(b_j)}} = \norm{[u(\varphi'(s_t))v(b_j)]_{\mu_j}} = |\rho_j([u(\varphi'(s_t))v(b_j)]_{\mu_j})|.
\end{equation*}
Observe that Proposition~\ref{prop: 2. existence of expansion} guarantees $\norm{[u(\varphi'(s_t))v(b_j)]_{\Lambda^{0+}(b_j)}} \neq 0$.
To prove the claim, we proceed by induction on $j$.
First, let $j=1$.
Applying \eqref{equation: 2. polynomial form} to our specific choices of $\mu_1$ and  $\rho_1$, we have
\begin{equation*}
\begin{split}
M_t &\geq |\rho_1([a(t)u(R_{s_t}(h))v]_{\mu_1})| \\
&= \left|\sum_{j'=1}^n \eta^{\mu_1(H_C)-b_{j'}} \exp\left( \left( \mu_1(H_\zeta)-\frac{1}{2}(\mu_1(H_C)-b_{j'}) \right)t \right)\rho_1([u(\mathbf{y}(h))v(b_{j'})]_{\mu_1})\right|.
\end{split}
\end{equation*}
The lowest power of $\eta$ in the polynomial inside the absolute value is $\mu_1(H_C)-b_1$. In view of \eqref{equation: 2. polynomial form of y}, the  coefficient of $\eta^{\mu_1(H_C)-b_1}$ is $\exp\left( \left( \mu_1(H_\zeta)-\frac{1}{2}(\mu_1(H_C)-b_{1}) \right)t \right)\rho_1([u(\varphi'(s_t))v(b_1)]_{\mu_1})$. Therefore, by Lemma~\ref{lemma: polynomial bound}, we have 
\begin{equation*}
\begin{split}
    M_t &\geq C_{d, I} \exp\left( \left( \mu_1(H_\zeta)-\frac{1}{2}(\mu_1(H_C)-b_{1}) \right)t \right)|\rho_1([u(\varphi'(s_t))v(b_1)]_{\mu_1})|\\
    &\geq C_{d,I} |\rho_1([u(\varphi'(s_t))v(b_1)]_{\mu_1})| \qquad \text{(because } \mu_1 \in \Lambda^{0+}(b_1) \text{ and } t > 0)\\
    &= C_{d,I}\norm{[u(\varphi'(s_t))v(b_1)]_{\Lambda^{0+}(b_1)}} \quad \, \, \text{(by the definition of }\mu_1 \text{ and } \rho_1)\\
    &\geq C_{d,I} D_1 \norm{v(b_1)} \qquad \qquad \qquad \quad \ \text{(by } \eqref{equation: 2. 0+bound}).
\end{split}
\end{equation*} 
Thus, the claim holds for $j=1$. 
Now, let $2 \leq m \leq n$ and assume that the claim holds for all $j = 1, \dots, m-1$. Similarly, applying \eqref{equation: 2. polynomial form} to our choices of $\mu_m$ and $\rho_m$, we have
\begin{equation*}
\begin{split}
M_t &\geq |\rho_m([a(t)u(R_{s_t}(h))v]_{\mu_m})| \\
&= \left|\sum_{j'=1}^n \eta^{\mu_m(H_C)-b_{j'}} \exp\left( \left( \mu_m(H_\zeta)-\frac{1}{2}(\mu_m(H_C)-b_{j'}) \right)t \right)\rho_m([u(\mathbf{y}(h))v(b_{j'})]_{\mu_m})\right|.
\end{split}
\end{equation*}
Consider the coefficient of the $\mu_m(H_C)-b_m$-degree term of the polynomial inside the absolute value. 
In view of \eqref{equation: 2. polynomial form of y}, the coefficient is 
\[ \exp\left(\left(\mu_m(H_\zeta)-\frac{1}{2}(\mu_m(H_C)-b_m) \right)t\right)\rho_m([u(\varphi'(s_t))v(b_m) +\sum_{j'=1}^{m-1}X_{b_{j'}-b_m}v(b_{j'})]_{\mu_m}).\]
By Lemma ~\ref{lemma: polynomial bound}, this implies that
\begin{equation*}
\begin{split}
     M_t &\geq C_{d,I}  \exp\left(\left(\mu_m(H_\zeta)-\frac{1}{2}(\mu_m(H_C)-b_m) \right)t\right)|\rho_m([u(\varphi'(s_t))v(b_m) +\sum_{j'=1}^{m-1}X_{b_{j'}-b_m}v(b_{j'})]_{\mu_m})|\\
     &\geq C_{d,I} |\rho_m([u(\varphi'(s_t))v(b_m) +\sum_{j'=1}^{m-1}X_{b_{j'}-b_m}v(b_{j'})]_{\mu_m})| \qquad \text{ (because }\mu_m \in \Lambda^{0+}(b_m)).
\end{split}
\end{equation*}
For each $1 \leq j' \leq m-1$, note that 
\begin{equation*}
\begin{split}
    \left|\rho_m([X_{b_{j'}-b_m} v(b_{j'})]_{\mu_m})\right| &\leq \norm{X_{b_{j'}-b_m}v(b_{j'})} \\
    & \leq \norm{X_{b_{j'}-b_m}}_{\mathrm{op}}\cdot \norm{v(b_{j'})}\\
    &\leq C_1 \norm{v(b_{j'})}\\
    & \leq C_1C_{d,I}^{-1}D_1^{-1}(1+C_1D_1^{-1})^{j'-1} M_t \qquad \text{ (by the induction hypothesis).}
\end{split}
\end{equation*}
Combining these inequalities, we have 
\[ M_t +\sum_{j'=1}^{m-1} C_1D_1^{-1}(1+C_1D_1^{-1})^{j'-1} M_t \geq C_{d,I} |\rho_m([u(\varphi'(s_t))v(b_m)]_{\mu_m})|.\]
Since the left-hand side of the above inequality is equal to $(1+C_1D_1^{-1})^{m-1}M_t$, this inequality implies 
\begin{equation*}
    \begin{split}
        M_t &\geq C_{d,I}(1+C_1D_1^{-1})^{-(m-1)}|\rho_m([u(\varphi'(s_t))v(b_m)]_{\mu_m})|\\
        &= C_{d,I} (1+C_1D_1^{-1})^{-(m-1)}\norm{[u(\varphi'(s_t))v(b_m)]_{\Lambda^{0+}(b_m)}}\\
        &\geq C_{d,I}(1+C_1D_1^{-1})^{-(m-1)}D_1\norm{v(b_m)}.
    \end{split}
\end{equation*}
Thus, the claim holds for $j = m$.
This completes the proof of the claim in ~\eqref{equation: induction hypo}.
Choose \[D_2 = C_{d,I}D_1(1+C_1D_1^{-1})^{-(n-1)}.\] 
Since the norm on $V$ is the sup-norm, by \eqref{equation: induction hypo}, we conclude that $M_t \geq D_2 \norm{v}$ for all $v \in V$ and $t \geq T$.

\end{proof}

\section{Non-divergence and unipotent invariance }\label{Sec: nondiv and unip inv}

\subsection{Non-divergence of \texorpdfstring{$\mu_{s_0,t}$}{mu(s0, t)}} 
Let $\mathcal{M}^1(X)$ be the space of Borel probability measures on $X = L/\Gamma$.
We need to show that the set $\{\mu_{s_0,t}:t> 0\}$ is contained in a compact subset of $\mathcal{M}^1(X)$ with respect to the weak-$\ast$ topology. 
To do this, it suffices to show the following non-divergence proposition. 

\begin{prop}\label{thm:non-divergence} There exists $T_1>0$ such that for every $\epsilon>0$, there exists a compact subset $K \subset X$ such that 
\[\mu_{s_0,t}(K) \geq 1-\epsilon \quad \text{for all } t\geq T_1.\]
\end{prop}

Before proving this proposition, we need the following definition and theorem. 

\begin{defn}
    Let $\mathfrak{L}$ be the Lie algebra of $L$ and let $d \in \N$. Let $\mathcal{P}_d(L)$ denote  the set of continuous maps $p: \R \to L$ such that for all $a, c \in \R$ and all $Y \in \mathfrak{L}$, the map \[\eta \in \R \mapsto \Ad(p(a\cdot \eta + c))(Y) \in \mathfrak{L}\] is a polynomial of degree at most $d$ in each coordinate of $\mathfrak{L}$.  
\end{defn}

Recall that to each Lie subgroup $H$ of $L$, we associate a nonzero vector $p_H \in \bigwedge\nolimits^{\dim H} \mathrm{Lie}(H)$. The following theorem combines and adapts the results from \cite{Sha96} for our purposes. 

\begin{thm}[cf.~\cite{Sha96}, Theorems 2.1 and 2.2]  \label{thm: nondivergence criterion} 
Let $M$ be the smallest closed normal subgroup of $L$ such that $\overline{L} \coloneqq L/M$ is a semisimple group with trivial center that does not contain any nontrivial compact normal subgroups. Define the quotient homomorphism $q: L \to \overline{L}$. Then there exist finitely many closed subgroups $W_1, \dots , W_r$ of $L$ such that each $W_j$ is of the form $q^{-1}(U_j)$ for some unipotent radical $U_j$ of a maximal parabolic subgroup of $\overline{L}$, $W_j\Gamma/\Gamma$ is compact in $X$, and the following holds: Given $d \in \N$ and $\epsilon$, $\alpha >0$, there exists a compact set $K \subset X$ such that for any polynomial map $p \in \mathcal{P}_d(L)$ and any bounded open interval $J \subset \R$, one of the following holds:
    \begin{enumerate}
        \item There exist $\gamma \in \Gamma$ and  $j \in \{1, \dots, r\}$ such that \[\sup_{\eta \in J} \norm{p(\eta)\gamma p_{W_j}} < \alpha.\] 
        \item \[\frac{1}{|J|} \left| \{\eta \in J: p(\eta)x_0 \in K\} \right| \geq 1 - \epsilon.\]
    \end{enumerate}
\end{thm}

\begin{proof}[Proof of Proposition~\ref{thm:non-divergence}]
    By Theorem \ref{thm: nondivergence criterion}, there exist closed subgroups $W_1, \dots, W_r$ of $L$ satisfying the conditions given in the theorem. Moreover, by Proposition \ref{prop: basic}, there exist $D_2>0$ and $T>0$ such that for every $\gamma \in \Gamma$, $j \in \{1, \dots, r\}$, and $t\geq T$, we have
\begin{equation*}
    \sup_{\eta \in I} \norm{a(t)u(R_{s_t}(\eta e^{-\frac{1}{2}t}))u(\varphi(s_t))\gamma p_{W_j}} \geq D_2\norm{u(\varphi(s_t))\gamma p_{W_j}}.
\end{equation*}
Since $\Gamma p_{W_j}$ is closed and discrete \cite[Theorem 3.4]{DM93}, $\norm{\gamma p_{W_j}}$ is uniformly bounded below by a positive constant for all $\gamma \in \Gamma$ and $1 \leq j \leq r$. Combining this with the fact that $s_t \to s_0$ as $t \to \infty$, for all $\gamma \in \Gamma$,  $j \in \{1, \dots, r\}$, and large enough $t$, say $t \geq T_1$, we obtain \begin{equation}\label{equation: 5. bdd below}
    \sup_{\eta \in I} \norm{a(t)u(R_{s_t}(\eta e^{-\frac{1}{2}t}))u(\varphi(s_t))\gamma p_{W_j}} \geq D_3
\end{equation} 
for some positive constant $D_3$.

On the other hand, note that for any $t$, the map $\eta\mapsto a(t)u(R_{s_t}(\eta e^{-\frac{1}{2}t}))u(\varphi(s_t))$ belongs to $\mathcal{P}_d(L)$ for some fixed $d \in \N$ independent of $t$. Due to \eqref{equation: 5. bdd below}, setting $\alpha = D_3$ rules out Case 1 of Theorem \ref{thm: nondivergence criterion}  for all $t \geq T_1$. This implies that for every $\epsilon>0$, there exists a compact set $K \subset X$ such that 
\begin{equation}\label{3. equation: measure nondivergent}
    \left|\{\eta \in I : a(t)u(R_{s_t}(\eta e^{-\frac{1}{2}t}))u(\varphi(s_t))x_0 \in K\}\right| \geq 1-\epsilon \quad \text{ for all } t \geq T_1.
\end{equation}
By the definition of $\mu_{s_0,t}$, this completes the proof. 
\end{proof}

\subsection{Unipotent invariance of limit measures}
Now, we fix a sequence $t_n \to \infty$.  By Proposition~\ref{thm:non-divergence}, after passing to a subsequence, we may assume that $\mu_{s_0, t_n}$ converges to $\mu_{s_0}$ for some $ \mu_{s_0} \in \M^1(X)$.
\begin{notation}
    Let $\varepsilon>0$. For $A, B \in \R$, if $|A-B|< \varepsilon$, we write $A \stackrel{\varepsilon}{\approx} B$.
    Fix a right $L$-invariant metric $d_L$ on $L$. For $x_1, x_2 \in X$ such that $x_2 = gx_1$ for some $g \in L$ with $d_L(g, e) < \varepsilon$, we write $x_1 \stackrel{\varepsilon}{\approx} x_2$. 
\end{notation}

Recall that $\zeta_1 \geq \zeta_2 \geq \cdots \geq \zeta_k = 1$. 
For each $i$ and $s \in I$, define 
\begin{equation*}
    \mathbf{w}_i(s) = (\varphi'_1(s), \varphi_2'(s), \dots, \varphi_i'(s), 0, \dots,0) \in \prod_{j=1}^k \mathbb{R}^{n_j-1}.
\end{equation*}
Let $k_1 \in \N$ be the smallest index $i \in \{1, 2, \dots, k-1\}$ such that $\zeta_{i}>\zeta_{i+1}$. If no such index exists (i.e., if $\zeta_1 = \zeta_2 = \cdots = \zeta_k$), we set $k_1 = k$. 

\begin{prop}\label{prop: unip inv}
  $\mu_{s_0}$ is invariant under the action of the unipotent subgroup $\{u(r\mathbf{w}_{k_1}(s_0)): r \in \R\}$. 
\end{prop}

\begin{proof}
    It suffices to show that for any $r \in \R$ and any $f \in C_c(X)$, 
\begin{equation} \label{equation: 4 unip inv WTS}
    \int_X f\big(u(r\mathbf{w}_{k_1}(s_0))y\big)d\mu_{s_0}(y) = \int_X f \,d\mu_{s_0}.
\end{equation}
Let $r \in \R$ and $f \in C_c(X)$. 
From the definition of $\mu_{s_0}$, we have 
\begin{align*}
    \int_X f\big(u(r\mathbf{w}_{k_1}(s_0))y\big)d\mu_{s_0}(y)
    =\lim_{n \to \infty}\int_I f \big(u(r\mathbf{w}_{k_1}(s_0)) a(t_n)u\big(R_{s_{t_n}}(\eta e^{-\frac{1}{2}t_n})\big)u(\varphi(s_{t_n}))x_0\big)d\eta. 
\end{align*}

Let $h = \eta e^{-\frac{1}{2}t_n}$. To simplify the notation, we shall write $s_n$ and $R(h)$ instead of $s_{t_n}$ and $R_{s_{t_n}}(h)$, respectively.
For each $1 \leq i \leq k$, let $R_i(h) \in \R^{n_i-1}$ denote the $i$-th component of $R(h)$, so that $R(h) = (R_1(h), \dots, R_k(h))$. Each component can be expressed as \[R_{i}(h) = h\varphi_i'(s_n)  + h^2P_{i}(h)\] where $P_{i}(h) = \sum_{j=2}^{l} \frac{\varphi_i^{(j)}(s_n)}{j!}h^{j-2}$ is a polynomial in $h$. 
In what follows, the big $O$-notation is taken as $n \to \infty$. For vectors, $O(\cdot)$ denotes a vector whose norm is $O(\cdot)$.

Now, we claim that
\begin{equation}\label{equation: 4 calc result unip inv}
    \begin{split}
        u(r\mathbf{w}_{k_1}(s_0))a(t_n)u\big(R(h)\big)= u(r(\mathbf{w}_{k_1}(s_0)-\mathbf{w}_{k_1}(s_n))+ O(e^{-\alpha t_n}))a(t_n)u(R(\tilde{h}))
    \end{split}
\end{equation} for some $\alpha>0$ and a suitable $\tilde{h}$ close to $h$.
To prove the claim, for each $i$, we first analyze the $i$-th component of 
$u(r\mathbf{w}_{k_1}(s_0))a(t_n)u\big(R(h)\big)$. Note that \[\pi_i(u(r\mathbf{w}_{k_1}(s_0))a(t_n)u\big(R(h)\big)) = \begin{cases} 
    u_i(r\varphi_i'(s_0))a_i(t_n)u_i(R_{i}(h)) & \text{ if } 1\leq i \leq k_1\\
     a_i(t_n)u_i(R_i(h)) & \text{ otherwise}
    \end{cases}.\]

Suppose that $1 \leq i \leq k_1$. 
Since $\zeta_i = \zeta_1$, we have
\begin{equation}\label{equation: 3. componentwise} 
    \begin{split}
        u_i(r\varphi_i'(s_0))a_i(t_n)u_i(R_{i}(h)) &= u_i\big(r\big(\varphi_i'(s_0)-\varphi_i'(s_n)\big)\big)a_i(t_n) u_i\big(re^{-\zeta_1 t_n}\varphi_i'(s_n) + R_{i}(h)\big)\\
        & = u_i\big(r\big(\varphi_i'(s_0)-\varphi_i'(s_n)\big)\big)a_i(t_n)u_i\big(\varphi'_i(s_n)\cdot(re^{-\zeta_1t_n} +h) + h^2P_{i}(h)\big).
    \end{split}
\end{equation}
Note that $re^{-\zeta_1 t_n} +h = re^{-\zeta_1 t_n} +\eta e^{-\frac{1}{2}t_n} = e^{-\frac{1}{2}t_n}(\eta +re^{(\frac{1}{2}-\zeta_1)t_n})$. By setting $\tilde{\eta} = \eta +re^{(\frac{1}{2}-\zeta_1)t_n}$ and $\tilde{h} = \tilde{\eta}e^{-\frac{1}{2}t_n}$, we can express \eqref{equation: 3. componentwise} as
\begin{equation}\label{equation: R and P polyns}
\begin{split}
    &u_i\big(r\big(\varphi_i'(s_0)-\varphi_i'(s_n)\big)\big)a_i(t_n)u_i\big(\varphi_i'(s_n)\tilde{h} + h^2P_{i}(h)\big) \\
    &= u_i\big(r\big(\varphi_i'(s_0)-\varphi_i'(s_n)\big)\big) a_i(t_n)u_i\big(R_{i}(\tilde{h})+ (h^2P_{i}(h)-\tilde{h}^2P_{i}(\tilde{h}))\big)
\end{split}    
\end{equation}

As in \cite[Claim 4.1]{ShahPYang24}, applying the mean value theorem yields
\[\norm{h^2P_i(h)-\tilde{h}^2P_i(\tilde{h})} \leq |h-\tilde{h}|\cdot\norm{2\overline{h}P_i(\overline{h}) + \overline{h}^2P_i'(\overline{h})} \] for some $\overline{h}$ between $h$ and $\tilde{h}$, where $\norm{\cdot}$ is the sup-norm on $\R^{n_i-1}$. Consequently,
\[
h^2P_{i}(h)-\tilde{h}^2P_{i}(\tilde{h}) = O(|h-\tilde{h}|e^{-t_n/2}) = O(e^{-\zeta_1t_n}\cdot e^{-t_n/2}) = O(e^{-(\zeta_1 + \frac{1}{2})t_n}).
\]
Thus, \eqref{equation: R and P polyns} can be written as\begin{equation*}\label{equation: result <k1}
    \begin{split}
         &u_i\big(r\big(\varphi_i'(s_0)-\varphi_i'(s_n)\big)\big)a_i(t_n)u_i\big(R_i(\tilde{h})+O(e^{-(\frac{1}{2}+\zeta_1)t_n})\big) \\
        & =  u_i\big(r\big(\varphi_i'(s_0)-\varphi_i'(s_n)\big)+O(e^{-\frac{1}{2}t_n})\big)a_i(t_n)u_i\big(R_i(\tilde{h})\big), \quad \text{ (because } \zeta_i = \zeta_1).
    \end{split}
\end{equation*}
Therefore, for $1 \leq i \leq k_1$,  we have
\begin{equation}\label{equation: i<k1}
\begin{split}
    \pi_i\big(u(r\mathbf{w}_{k_1}(s_0))a(t_n)u\big(R(h)\big)\big) 
     &=   u_i\big(r\big(\varphi_i'(s_0)-\varphi_i'(s_n)\big)+O(e^{-\frac{1}{2}t_n})\big)a_i(t_n)u_i\big(R_i(\tilde{h})\big).
\end{split}
\end{equation}

On the other hand, if $k_1 <k$,
suppose that $k_1<i\leq k$. The $i$-th component of $u(r\mathbf{w}_{k_1}(s_0))a(t_n)u\big(R(h)\big)$ can be expressed as
\begin{equation}\label{equation: >k1}
    \begin{split}
        a_i(t_n)u_i(R_i(h)) 
         = a_i(t_n)u_i\big(R_i(\tilde{h})+(R_i(h)-R_i(\tilde{h}))\big)
    \end{split}
\end{equation}
Similarly, applying the mean value theorem, we have \[\norm{R_i(h) -R_i(\tilde{h})} \leq |h-\tilde{h}|\norm{\varphi_i'(s_n) + 2\hat{h}P_i(\hat{h}) + \hat{h}^2 P_i'(\hat{h})}\] for some $\hat{h}$ between $h$ and $\tilde{h}$. It follows that \[R_i(h) - R_i(\tilde{h}) = O(|h-\tilde{h}|) = O(e^{-\zeta_1 t_n}).\]
This implies that (\ref{equation: >k1}) can be written as
\begin{equation*}\label{equation: result >k1}
    \begin{split}
        a_i(t_n)u_i\big(R_i(\tilde{h}) + O(e^{-\zeta_1 t_n})\big)
        & = u_i\big(O(e^{(\zeta_i-\zeta_1)t_n})\big)a_i(t_n)u_i\big(R_i(\tilde{h})\big).
    \end{split}
\end{equation*}
Therefore, for $k_1 < i \leq k$, we have
\begin{equation}\label{equation: i>k1}
    \begin{split}
        \pi_i\big(u(r\mathbf{w}_{k_1}(s_0))a(t_n)u\big(R(h)\big)\big)
        = u_i\big(O(e^{-(\zeta_1-\zeta_i)t_n})\big)a_i(t_n)u_i\big(R_i(\tilde{h})\big)
    \end{split}
\end{equation}

For notational convenience, if $k_1 = k$, then we set $\zeta_{k_1+1} = 0$. Now, let \[\alpha = \min\{\frac{1}{2}, \zeta_1-\zeta_{k_1+1}\}>0.\] 
Combining \eqref{equation: i<k1} with \eqref{equation: i>k1}, we conclude that the claim in  \eqref{equation: 4 calc result unip inv} holds. 

Let $\varepsilon>0$. Since $f$ is uniformly continuous, there exists $\delta>0$ such that for any $y,z \in X$ with $y \stackrel{\delta}{\approx} z $, we have $f(y) \stackrel{\varepsilon}{\approx}f(z)$. 
 Note that since $\varphi'(s_n) \to \varphi'(s_0)$,  for any sufficiently large $n$,
\begin{equation}\label{equation: 4. uniform conti}
    d_L\big(u\big(r(\mathbf{w}_{k_1}(s_0)-\mathbf{w}_{k_1}(s_n))+ O(e^{-\alpha t_n})\big), e\big) < \delta.
\end{equation}
Then, by the change of variable from $\eta$ to $\tilde{\eta}$, for any sufficiently large $n$,
we have
\begin{equation*}
    \begin{split}
       & \int_I f \big(u(r\mathbf{w}_{k_1}(s_0))a(t_n)u(R(h))u(\varphi(s_n))x_0\big)d\eta \\
        & \qquad \ \  = \int_{I + re^{(\frac{1}{2}-\zeta_1)t_n}} f\Big(u\big(r(\mathbf{w}_{k_1}(s_0)-\mathbf{w}_{k_1}(s_n))+ O(e^{-\alpha t_n})\big)a(t_n) u\big(R(\tilde{h})\big)u(\varphi(s_n))x_0\Big) d\tilde{\eta} \quad (\text{by }\eqref{equation: 4 calc result unip inv})\\
        & \qquad \ \,\, \stackrel{\varepsilon}{\approx} \int_{I + re^{(\frac{1}{2}-\zeta_1)t_n}} f\big(a(t_n)u\big(R(\tilde{h})\big)u(\varphi(s_n))x_0\big) d\tilde{\eta} \quad \text{ (by (\ref{equation: 4. uniform conti})})\\
        & \stackrel{2|r|e^{(\frac{1}{2}-\zeta_1)t_n}\norm{f}_\infty}{\approx}\int_I f\big(a(t_n)u\big(R(\tilde{h})\big)u(\varphi(s_n))x_0\big)d\tilde{\eta}.
    \end{split}
\end{equation*}

Therefore, since $\frac{1}{2}-\zeta_1<0$ and $\varepsilon$ is arbitrary, letting $n\to \infty$ yields (\ref{equation: 4 unip inv WTS}). 
\end{proof}

\section{Equidistribution of limit measures}\label{Sec: Equidistribution} 
In this section, we identify obstructions to equidistribution using the linearization technique. We then complete the proof of Theorem~\ref{thm: 2. equidistribution of polyn approx} and derive Theorems~\ref{Thm : 1. shrinking equidistribution at point} and~\ref{Theorem : 1. main theorem}.
More precisely, the linearization argument will show that if \(\mu_{s_0}\) is not \(L\)-invariant, then there exists \(H\in\mathscr H\) such that
\[
u(\varphi(s_0))p_H\in V_L^{0-}(A)
\] and \(G\) does not fix \(p_H\). This will be proved in more detail later in this section. Therefore, we first analyze the obstruction sets
\[
\Delta_H=\{g\in G:gp_H\in V_L^{0-}(A)\}
\] or, when convenient, their restrictions to the expanding horospherical subgroup \(N\), which, via the isomorphism $
u:\prod_{i=1}^k\mathbb R^{n_i-1}\longrightarrow N$, 
are identified with
\[
S_H=
\left\{
\mathbf x\in\prod_{i=1}^k\mathbb R^{n_i-1}:
u(\mathbf x)p_H\in V_L^{0-}(A)
\right\}.
\]
We will consider two cases: when the orbit $Gp_H$ is closed in $V_L$ and when it is not.

For each $n\geq 2$, let $Q_n$ denote the quadratic form in
\eqref{eq: quad form} on $\R^{n+1}$. Before analyzing each case, we
record the following extension lemma.

\begin{lem}\label{lemma:representation-extension}
Let $\mathbf G=\mathbf{SO}(Q_n)$ be the algebraic group over $\R$, and let
$\sigma:\mathbf G(\R)^\circ\to\mathrm{GL}(V)$ be a finite-dimensional
representation over $\R$. Then there exists a unique rational
representation
\[
\overline{\sigma}:\mathbf G\longrightarrow\mathrm{GL}(V)
\]
over $\R$ whose restriction to $\mathbf G(\R)^\circ$ is $\sigma$.
\end{lem}

\begin{proof}
Let
$q:\widetilde{\mathbf G}:=\operatorname{Spin}(Q_n)\to\mathbf G$ be a quotient map of
algebraic groups whose kernel is the finite central subgroup of scalar
elements defined by $z^2=1$ (see \cite[p.~533]{MilneAG}). By \cite[\text{Theorem }22.53]{MilneAG}, the Lie algebra representation \(d\sigma\circ dq\) integrates uniquely to a rational representation
\[
\widetilde{\sigma}:\widetilde{\mathbf G}\longrightarrow\operatorname{GL}(V).
\]
By Cartan’s theorem \cite[\text{Theorem }25.54]{MilneAG}, \(\widetilde{\mathbf G}(\mathbb R)\) is connected. Since \(q_{\mathbb R}\) is a local diffeomorphism, its image is \(\mathbf G(\mathbb R)^\circ\). Hence the homomorphisms \(\widetilde{\sigma}_{\mathbb R}\) and \(\sigma\circ q_{\mathbb R}\), having the same differential and connected domain, are equal. In particular, \(\widetilde{\sigma}\) is trivial on \(\ker q\), and therefore factors uniquely through $q$ to a rational representation \(\overline{\sigma}\) of \(\mathbf G\) by \cite[\text{Theorem }5.13]{MilneAG}.
\end{proof}

From now on, let
\[
\mathbf G
=
\prod_{i=1}^k\mathbf{SO}(Q_{n_i}).
\]
Then $G=\mathbf G(\R)^\circ$ and \(G\) has finite index in \(\mathbf G(\mathbb R)\).
For each $i$, apply Lemma~\ref{lemma:representation-extension} to the
restriction of the $G$-action on $V_L$ to $G_i$. This gives a unique
rational extension to $\mathbf{SO}(Q_{n_i})$. Since the original
$G_i$-actions commute pairwise and $G_i\times G_j$ is Zariski dense in
$\mathbf{SO}(Q_{n_i})\times\mathbf{SO}(Q_{n_j})$, their rational
extensions also commute pairwise.
Hence the \(G\)-action on \(V_L\) extends to a rational action of
\(\mathbf G\).

\subsection{Case 1: \texorpdfstring{$Gp_H$}{GpH} is closed}

For each $i$, let $\mathcal{W}_i = N_{G_i}(A_i)/Z_{G_i}(A_i)$ be the Weyl group of $G_i$. We can choose a coset representative $w_i \in N_{G_i}(A_i) \subset G_i$ of an element of $\mathcal{W}_i$ such that $w_i = w_i^{-1}$ and $w_i a_i(t) w_i^{-1} = a_i(-t)$ for all $t \in \R$. Define $\mathcal{W}'$ as the subgroup of $G$ generated by $\{w_1, \dots, w_k\}$.

For each $\mathcal{B}\subseteq\{1,2,\dots,k\}$, we define an element $w_\mathcal{B} \in \mathcal{W}'$ whose $i$-th component is $w_i$ if $i \in \mathcal{B}$, and $e$ if $i \notin \mathcal{B}$. Furthermore, we define $G_\mathcal{B} = \prod_{i \in \mathcal{B}} G_i$ and the corresponding projection map $\pi_{\mathcal{B}}: G \to G_\mathcal{B}$. We also define the unipotent subgroup $N_{\mathcal{B}} = \pi_{\mathcal{B}}(N) =  \prod_{i \in \mathcal{B}} N_i$, with the convention that $N_\emptyset = \{e\}$. We identify $G_{\mathcal B}$ with its standard coordinate subgroup of $G$ whenever no confusion can arise.

Recall that for each $H \in \mathscr{H}$, we define
\begin{equation*}
\begin{split}
        N_G^1(H) = \{g \in N_G(H): \det(\Ad\, g|_{\text{Lie}(H)}) = 1  \}.
\end{split}
\end{equation*}

\begin{defn}
    Let $\mathcal{J}\subset \{1, 2, \dots, k\}$ be a nonempty set of indices and let $m_{\mathcal{J}} \in \N$ satisfy $1 \leq m_{\mathcal{J}} \leq \min_{j \in \mathcal{J}} n_j$. For each $j \in \mathcal{J}$, choose $\theta_j \in \mathrm{Aut}(\SO(m_{\mathcal{J}}, 1))$. We call the subgroup
    \[
    \left\{\big(\theta_j(h)\big)_{j \in \mathcal{J}}:h \in \SO(m_\mathcal{J},1)\right\}
    \subseteq \prod_{j \in \mathcal{J}}G_j
    \]
    a generalized diagonal embedding of $\mathrm{SO}(m_\mathcal{J}, 1)$, where each $\theta_j(h)$ is viewed as an element of $G_j$ via the standard inclusion $\mathrm{SO}(m_{\mathcal{J}},1) \hookrightarrow G_j = \mathrm{SO}(n_j,1)$. We denote such a subgroup by $\Delta_{\mathcal{J}}\mathrm{SO}(m_\mathcal{J},1)$, or simply by $\Delta \mathrm{SO}(m_\mathcal{J}, 1)$.
\end{defn}

\begin{prop}\label{prop:closed-orbit-obstruction}
Let \(H\in\mathscr H\) be such that \(Gp_H\) is closed and
\(G\) does not fix \(p_H\). Suppose that
\[
\Delta_H
=
\left\{
g\in G: gp_H\in V_L^{0-}(A)
\right\}
\]
is nonempty. Then there exist \(\xi\in G\) and a proper reductive subgroup \(F\) of \(G\) containing \(A\)
such that the following conditions hold.
\begin{enumerate}

\item
\[
F=N_G^1\bigl(\xi H\xi^{-1}\bigr).
\]

\item
There exist a possibly empty subset
\[
\mathcal J_{\mathrm{tor}}\subseteq\{1,\dots,k\}
\]
and a possibly empty partition
\[
\mathscr P=\{\mathcal J_1,\dots,\mathcal J_p\}
\]
of \(\{1,\dots,k\}\setminus\mathcal J_{\mathrm{tor}}\), where
\[
0\leq p\leq k-|\mathcal J_{\mathrm{tor}}|,
\]
such that
\[
F^\circ=S_0\cdot S_{\mathrm{tor}}\cdot S_1\cdots S_p
\]
is an almost direct product. Here \(F^\circ\) denotes the identity
component of \(F\), \(S_0\) is the maximal connected compact normal
subgroup of \(F^\circ\), and \(S_{\mathrm{tor}}\) is the maximal
connected \(\mathbb R\)-split subtorus of \(Z(F^\circ)\). Moreover,
\[
A_{\mathrm{tor}}
:=\pi_{\mathcal J_{\mathrm{tor}}}(A)
\subseteq S_{\mathrm{tor}}
\subseteq
\prod_{i\in\mathcal J_{\mathrm{tor}}}A_i.
\]
For every \(1\leq j\leq p\),
\[
S_j
=
\Delta_{\mathcal J_j}\mathrm{SO}(m_{\mathcal J_j},1),
\qquad
2\leq m_{\mathcal J_j}
\leq \min_{i\in\mathcal J_j}n_i.
\]
Furthermore,
\[
\zeta_{i_1}=\zeta_{i_2}
\qquad
\text{for all }i_1,i_2\in\mathcal J_j
\]
for every \(1\leq j\leq p\). The component group \(F/F^\circ\)
is finite. We use the conventions that
\(S_{\mathrm{tor}}=\{e\}\) and \(A_{\mathrm{tor}}=\{e\}\) when
\(\mathcal J_{\mathrm{tor}}=\emptyset\). When \(p=0\), we use the
conventions
\[
\mathscr P=\emptyset,
\qquad
\mathcal J_{\mathrm{tor}}=\{1,\dots,k\},
\qquad
S_1\cdots S_p=\{e\}.
\]

\item
Define the finite set
\[
\mathcal{W}_{\mathrm{tor}}
:=
\left\{
 w_{\mathcal B} \in \mathcal{W}':
 \mathcal B\subseteq\mathcal J_{\mathrm{tor}},\quad
 w_{\mathcal B}^{-1}A_{\mathrm{tor}}w_{\mathcal B}
 \subseteq S_{\mathrm{tor}}
\right\}.
\]
Then
\[
\Delta_H=P^-F \mathcal{W}_{\mathrm{tor}}\xi.
\]
Moreover, 
\[ P^- F \mathcal{W}_{\mathrm{tor}} = P^- F^\circ \mathcal{W}_{\mathrm{tor}}.
\]

\end{enumerate}
\end{prop}

Before proving the proposition, we deduce one of its main consequences for equidistribution: under the assumptions of Theorem~\ref{Theorem : 1. main theorem}, the exceptional set \(E_1\) arising from closed orbits is null.

Recall that
\[
\begin{aligned}
	I_H
	&=\{s\in I:u(\varphi(s))\in\Delta_H\},\\
	\mathscr H_1
	&=\{H\in\mathscr H:
	Gp_H\text{ is closed in }V_L
	\text{ and }G\text{ does not fix }p_H\},\\
	E_1
	&=\bigcup_{H\in\mathscr H_1}I_H.
\end{aligned}
\]
Recall also that \(\mathscr H_0\) consists of those \(H\in\mathscr H_1\)
for which \(\Delta_H\neq\varnothing\) and the maximal connected
\(\mathbb R\)-split subtorus of the center of \(N_G^1(H)^\circ\) is trivial.

\begin{cor}\label{cor:closed-orbit-nullity} 
	Under the assumptions of Theorem~\ref{Theorem : 1. main theorem}, the set
	\(I_H \) is null for every \(H\in\mathscr H_1\). Consequently, \(E_1\) is
	a null set.
\end{cor}

\begin{proof}[Proof of Corollary~\ref{cor:closed-orbit-nullity}]
	Let \(H\in\mathscr H\) be such that \(Gp_H\) is closed in \(V_L\),
	\(G\) does not fix \(p_H\), and \(\Delta_H\neq\varnothing\).
	If the maximal connected \(\mathbb R\)-split subtorus of the center of the
	identity component of \(N_G^1(H)\) is nontrivial, then $I_H$
	is a null set. Indeed, with the notation of Proposition~\ref{prop:closed-orbit-obstruction},
	in this case \(\mathcal J_{\mathrm{tor}}\neq\emptyset\).
	Fix \(i_0\in\mathcal J_{\mathrm{tor}}\). By part~(3) of the proposition,
	\(\varphi_{i_0}(s)\) takes only finitely many values as
	\(s\) ranges over \(I_H\) because $u_{i_0}(\varphi_{i_0}(s)) \in P^-_{i_0}\pi_{i_0}(\mathcal W_\mathrm{tor}\xi)$ and $N_{i_0} \cap P^-_{i_0} = \{e\}$. Since
	\(\varphi_{i_0}'(s)\neq0\) almost everywhere, every level set of
	\(\varphi_{i_0}\) is null by a Lebesgue density point argument. Consequently, $I_H$ is null.
	If the maximal connected \(\mathbb R\)-split subtorus of the center of the
	identity component of \(N_G^1(H)\) is trivial, then \(H\in\mathscr H_0\),
	so \(I_H\) is null by the hypothesis of Theorem~\ref{Theorem : 1. main theorem}.
	Since \(\mathscr H_1\) is countable, \(E_1\) is also null.
\end{proof}

\begin{proof}[Proof of Proposition \ref{prop:closed-orbit-obstruction}]
Fix \(g_*\in\Delta_H\). Since
\(g_*p_H\in V_L^{0-}(A)\) and \(Gp_H\) is closed, there exists \(\xi_0\in G\) such
that
\begin{equation*}
\lim_{t\to\infty}a(t)g_*p_H=\xi_0p_H.
\end{equation*}
Hence
\[
A\subseteq
\operatorname{Stab}_G(\xi_0p_H)
=
N_G^1\bigl(\xi_0H\xi_0^{-1}\bigr).
\]
Recall that the \(G\)-action on \(V_L\) extends to a rational \(\mathbf G\)-action.
Set
\[
\mathbf F_0:=\operatorname{Stab}_{\mathbf G}(\xi_0p_H),
\qquad
F_0:=\mathbf F_0(\mathbb R)\cap G
=\operatorname{Stab}_G(\xi_0p_H).
\]
Since \(
G(\xi_0p_H)=Gp_H
\)
is closed and since \(G\) has finite index in \(\mathbf G(\mathbb R)\), the
orbit \(\mathbf G(\mathbb R)(\xi_0p_H)\) is a finite union of translates
of \(G(\xi_0p_H)\), and hence is closed. Matsushima's criterion therefore
implies that \(\mathbf F_0\) is reductive. In particular, \(F_0\) is a
real reductive subgroup of \(G\).
Moreover, \(F_0\) is proper, for otherwise \(G\) would fix \(p_H\).

Since \(A\) is connected, \(A\subseteq F_0^\circ\). Therefore, for
every \(i\), the connected reductive subgroup
\(\pi_i(F_0^\circ)\) contains \(A_i\). Then, there exist \(1\leq m_i\leq n_i\),
\(z_i\in Z_{G_i}(A_i)\), and a compact subgroup \(C_i\subseteq Z_{G_i}(\mathrm{SO}(m_i,1))\) such
that
\[
\pi_i(F_0^\circ)
=
z_i^{-1}\bigl(\mathrm{SO}(m_i,1)\cdot C_i\bigr)z_i.
\]
Here and below, \(\mathrm{SO}(m_i,1)\) denotes the standard subgroup of \(G_i=\mathrm{SO}(n_i,1)\).
Set
\[
z=(z_1,\dots,z_k),\qquad
\xi=z\xi_0
\]
and let
\[
\mathbf F
:=z\mathbf F_0z^{-1}
=\operatorname{Stab}_{\mathbf G}(\xi p_H),
\qquad
F:=\mathbf F(\mathbb R)\cap G
=N_G^1(\xi H\xi^{-1})
=\operatorname{Stab}_G(\xi p_H).
\]
Since \(z\in Z_G(A)\), we still have \(A\subseteq F^\circ\), and
\[
\pi_i(F^\circ)=\mathrm{SO}(m_i,1)\cdot C_i
\]
for every \(i\). Since \(\mathbf F(\mathbb R)\) has finitely many
connected components, so does \(F=\mathbf F(\mathbb R)\cap G\). Thus
\(F/F^\circ\) is finite.

We now describe the structure of \(F^\circ\). Let \(S_1,\dots,S_p\)
be the noncompact simple normal factors of \(F^\circ\), and let
\(S_{\mathrm{tor}}\) be the maximal connected \(\mathbb R\)-split
subtorus of \(Z(F^\circ)\). Let \(S_0\) be the maximal connected
compact normal subgroup of \(F^\circ\). Then
\begin{equation*}
F^\circ=S_0\cdot S_{\mathrm{tor}}\cdot S_1\cdots S_p
\end{equation*}
is an almost direct product.

For each \(j\), define the coordinate support of \(S_j\) by
\[
\mathcal J_j:=\{i:\pi_i(S_j)\neq\{e\}\}.
\]
Since \(S_j\) is simple, every nontrivial coordinate projection of
\(S_j\) is injective. The sets \(\mathcal J_j\) are pairwise
disjoint. Indeed, if two distinct noncompact simple normal factors had
nontrivial projection to the same coordinate, their images would be two
commuting noncompact normal subgroups of the unique noncompact simple
factor $\mathrm{SO}(m_i,1)$ of \(\pi_i(F^\circ)\), which is impossible. Therefore, for each $i \in \mathcal{J}_j$, $\pi_i(S_j) \cong \mathrm{SO}(m_i,1) $.  It follows that
\[
S_j=\Delta_{\mathcal J_j}\mathrm{SO}(m_{\mathcal J_j},1) = \{(\pi_i\circ(\pi_{i_0}|_{S_j})^{-1}(h))_{i \in \mathcal{J}_j} \in G_{\mathcal{J}_j} : h \in \SO(m_{\mathcal{J}_j},1)\}
\]
for some \(m_{\mathcal J_j}\geq2\) and some $i_0 \in \mathcal{J}_j$.

Let
\[
\mathcal J_{\mathrm{tor}}
=
\{1,\dots,k\}\setminus\bigcup_{j=1}^p\mathcal J_j.
\]
If \(p=0\), then \(\mathscr P=\emptyset\) and
\(\mathcal J_{\mathrm{tor}}=\{1,\dots,k\}\). If $p \neq 0$, then the projection of \(S_{\mathrm{tor}}\) to any coordinate indexed by \(i\in\mathcal J_j\) is trivial, because the centralizer of the noncompact factor \(\mathrm{SO}(m_{\mathcal J_j},1)\) in \(G_i\) is compact. Hence
\[
S_{\mathrm{tor}}
\subseteq
\prod_{i\in\mathcal J_{\mathrm{tor}}}A_i.
\]
Since \(A\subseteq F^\circ\), we also have
\[
A_{\mathrm{tor}}=
\pi_{\mathcal J_{\mathrm{tor}}}(A)
\subseteq S_{\mathrm{tor}}.
\]

Note that $\pi_{\mathcal{J}_j}(A) \subset S_j$ and the coordinate projections of \(S_j\) are related by automorphisms of
\(\mathrm{SO}(m_{\mathcal J_j},1)\), each of which is induced by
conjugation by an element of \(\mathrm O(m_{\mathcal J_j},1)\);
see~\cite[Theorem~11.9(7)]{BCLS15}. Since conjugation preserves
eigenvalues, it follows that
\[
\zeta_{i_1}=\zeta_{i_2}
\qquad
(i_1,i_2\in\mathcal J_j).
\]
This proves items~(1) and~(2).

For the rest of the proof, if
\(\mathcal B\subseteq\{1,\dots,k\}\), we retain the notation
\(G_{\mathcal B}\), \(\pi_{\mathcal B}\), and \(N_{\mathcal B}\)
introduced above. Let
\[
\mathbf G_{\mathcal B}
:=\prod_{i\in\mathcal B}\mathbf{SO}(Q_{n_i}),
\qquad
\boldsymbol\pi_{\mathcal B}:\mathbf G\longrightarrow\mathbf G_{\mathcal B}
\]
be the corresponding algebraic group and algebraic projection, and set
\[
\mathbf F_{\mathcal B}:=\boldsymbol\pi_{\mathcal B}(\mathbf F),
\qquad
F_{\mathcal B}:=\pi_{\mathcal B}(F),
\quad
A_{\mathcal B}=\pi_{\mathcal B}(A),
\quad
a_{\mathcal B}(t)=\pi_{\mathcal B}(a(t)).
\]
For each \(i\in\mathcal B\), let \(\mathbf N_i\) be the
\(\mathbb R\)-defined unipotent algebraic subgroup parametrized by the
same polynomial formula as \(u_i(\mathbf x)\), and put
\[
\mathbf N_{\mathcal B}:=\prod_{i\in\mathcal B}\mathbf N_i.
\]
Then \(\mathbf N_{\mathcal B}(\mathbb R)=N_{\mathcal B}\). Since
\(\mathbf F_{\mathcal B}\) is the algebraic image of the reductive group
\(\mathbf F\), it is reductive.
The surjective algebraic homomorphism
\(\boldsymbol\pi_{\mathcal B}:\mathbf F\to\mathbf F_{\mathcal B}\)
has surjective differential, so \(F_{\mathcal B}\) is an open and closed
subgroup of \(\mathbf F_{\mathcal B}(\mathbb R)\).

 Define
\[
\Omega_{\mathrm{tor}}
:=
\left\{
 w_{\mathcal B}:\mathcal B\subseteq\mathcal J_{\mathrm{tor}},\quad
 w_{\mathcal B}S_{\mathrm{tor}}w_{\mathcal B}^{-1}
 =S_{\mathrm{tor}}
\right\}.
\]
To show $P^-F\mathcal{W}_{\mathrm{tor}}=P^-F^\circ \mathcal{W}_{\mathrm{tor}}$, we need the following claim.

\medskip
\noindent
\textbf{Claim 1.}
Every \(f\in F\) can be written as
\begin{equation*}
f=z_fs_fw_{\mathcal B_f},
\end{equation*}
where
\[
z_f\in Z_G(A),
\qquad
s_f\in S_1\cdots S_p,
\qquad
w_{\mathcal B_f}\in\Omega_{\mathrm{tor}}.
\]

Indeed, \(F^\circ\) is normal in \(F\). Hence conjugation by \(f\)
permutes the noncompact simple normal factors of \(F^\circ\). Since
conjugation in \(G=\prod_iG_i\) is componentwise, we have
\[
fS_jf^{-1}=S_j
\qquad
(1\leq j\leq p).
\]
Since $A_{\mathcal{J}_j}$ is a maximal \(\mathbb R\)-split torus of \(S_j\), after
multiplying \(f\) on the right by a suitable element of \(S_j\), we
may assume that \(f\) normalizes \(A_{\mathcal J_j}\). If the resulting
action on \(A_{\mathcal J_j}\) is inversion, we multiply once more by a
Weyl representative in \(S_j\). Consequently, for every
\(1\leq j\leq p\), there is \(r_j\in S_j\) such that \(fr_j\)
centralizes \(A_{\mathcal J_j}\). Put
\[
r=r_1\cdots r_p,
\qquad
s_f:=r^{-1}.
\]
Furthermore, \(f\) also normalizes \(S_{\mathrm{tor}}\). Hence there exists
\(w_{\mathcal B_f}\in \Omega_{\mathrm{tor}}\) such that
\[
z_f:=frw_{\mathcal B_f}^{-1}\in Z_G(A).
\]
Since \(r\in G_{\cup_j\mathcal J_j}\) and
\(w_{\mathcal B_f}\in G_{\mathcal J_{\mathrm{tor}}}\), they
commute, and therefore
\[
f=z_fs_fw_{\mathcal B_f},
\]
which proves Claim~1.

\medskip
\noindent
\textbf{Claim 2.}
For every \(\mathcal B\subseteq\mathcal J_{\mathrm{tor}}\),
\begin{equation}\label{eq:SH-move-W-right}
P^-w_{\mathcal B}F^\circ=P^-F^\circ w_{\mathcal B}.
\end{equation}
Moreover,
\begin{equation}\label{eq:SH-disconnected-absorbed}
P^-F\mathcal{W}_{\mathrm{tor}}=P^-F^\circ \mathcal{W}_{\mathrm{tor}}.
\end{equation}

To prove \eqref{eq:SH-move-W-right}, note that \(S_0\) centralizes
\(A\), and hence \(S_0\subseteq Z_G(A)\subseteq P^-\). Also,
\(S_{\mathrm{tor}}\subseteq\prod_{i\in\mathcal J_{\mathrm{tor}}}A_i
\subseteq P^-\), while \(w_{\mathcal B}\) commutes with
\(S_1\cdots S_p\). Since \(w_{\mathcal B}\) normalizes
\(Z_G(A)\) and the full torus
\(\prod_{i\in\mathcal J_{\mathrm{tor}}}A_i\), both sides of
\eqref{eq:SH-move-W-right} are equal to
\[
P^-S_1\cdots S_pw_{\mathcal B}.
\]

By Claim~1,
\[
F\subseteq P^-F^\circ\Omega_{\mathrm{tor}}.
\]
Since \(A_{\mathrm{tor}}\subseteq S_{\mathrm{tor}}\), we have
\(\Omega_{\mathrm{tor}}\subseteq \mathcal{W}_{\mathrm{tor}}\). Moreover,
\[
\mathcal{W}_{\mathrm{tor}}\Omega_{\mathrm{tor}}=\mathcal{W}_{\mathrm{tor}}.
\]
Indeed, if \(w\in \mathcal{W}_{\mathrm{tor}}\) and
\(\omega\in\Omega_{\mathrm{tor}}\), then
\[
(w\omega)^{-1}A_{\mathrm{tor}}(w\omega)
=
\omega^{-1}(w^{-1}A_{\mathrm{tor}}w)\omega
\subseteq S_{\mathrm{tor}}.
\]
Since the elements of \(\mathcal{W}'\) commute with one another by construction, the same equality holds with
\(\Omega_{\mathrm{tor}}\) on the left. It follows that
\[
P^-F\mathcal{W}_{\mathrm{tor}}
\subseteq
P^-F^\circ\Omega_{\mathrm{tor}}\mathcal{W}_{\mathrm{tor}}
=
P^-F^\circ \mathcal{W}_{\mathrm{tor}}.
\]
The reverse inclusion follows from \(F^\circ\subseteq F\), proving
\eqref{eq:SH-disconnected-absorbed}.

We shall also use the following observation.

\medskip
\noindent
\textbf{Claim 3.}
Let \(\mathcal B\subseteq\{1,\dots,k\}\), set
\(\mathfrak n_{\mathcal B}:=\operatorname{Lie}(N_{\mathcal B})\), and let
\(X\in\mathfrak n_{\mathcal B}\). If
\[
a_{\mathcal B}(t_n)\exp(X)F_{\mathcal B}
\]
converges in \(G_{\mathcal B}/F_{\mathcal B}\) for some sequence
\(t_n\to\infty\), then \(\exp(X)\in F_{\mathcal B}\).

Indeed, \(\mathbf F_{\mathcal B}\) is reductive, so the algebraic quotient
\(\mathbf G_{\mathcal B}/\mathbf F_{\mathcal B}\) is affine, again by Matsushima's criterion. 
Since unipotent algebraic groups have closed orbits on affine varieties
\cite[Theorem~17.64]{MilneAG},
\(N_{\mathcal B}F_{\mathcal B}\) is closed in
\(G_{\mathcal B}\). Consequently, the natural map
\[
N_{\mathcal B}/(N_{\mathcal B}\cap F_{\mathcal B})
\longrightarrow
N_{\mathcal B}F_{\mathcal B}/F_{\mathcal B}
\]
is a homeomorphism onto a closed orbit in \(G_{\mathcal B}/F_{\mathcal B}\).
Since \(a_{\mathcal B}(t_n)\in F_{\mathcal B}\),
\[
a_{\mathcal B}(t_n)\exp(X)F_{\mathcal B}
=
a_{\mathcal B}(t_n)\exp(X)a_{\mathcal B}(-t_n)F_{\mathcal B}
=
\exp\!\bigl(\operatorname{Ad}(a_{\mathcal B}(t_n))X\bigr)
F_{\mathcal B}.
\]
If \(\exp(X)\notin F_{\mathcal B}\), then
\[
\exp\!\bigl(\operatorname{Ad}(a_{\mathcal B}(t_n))X\bigr)
F_{\mathcal B}\to\infty
\quad\text{in }G_{\mathcal B}/F_{\mathcal B},
\]
contrary to the assumed convergence. Hence
\(\exp(X)\in F_{\mathcal B}\). This proves Claim~3.

For later use, we also note that
\begin{equation}\label{eq:SH-full-block-Weyl}
w_{\mathcal J_j}\in P^-F^\circ
\qquad
(1\leq j\leq p).
\end{equation}
Indeed, choose a nontrivial Weyl representative
\(\widehat w_{\mathcal J_j}\in N_{S_j}(A_{\mathcal J_j})\).
For every \(i\in\mathcal J_j\), both
\(\pi_i(\widehat w_{\mathcal J_j})\) and \(w_i\) invert \(A_i\). Hence
\[
w_{\mathcal J_j}\widehat w_{\mathcal J_j}^{-1}\in Z_G(A),
\]
and therefore
\[
w_{\mathcal J_j}\in Z_G(A)S_j\subseteq P^-F^\circ.
\]

We need the following claim before proving item~(3).

\medskip
\noindent
\textbf{Claim 4.}
Let \(\mathcal B\subseteq\mathcal J_{\mathrm{tor}}\). If
\[
a(t_n)w_{\mathcal B}F
\]
converges in \(G/F\) for some \(t_n\to\infty\), then
\(w_{\mathcal B}\in \mathcal{W}_{\mathrm{tor}}\).

Since \(F/F^\circ\) is finite, the natural map
\[
G/F^\circ\longrightarrow G/F
\]
is a finite covering and hence proper. Thus, after passing to a
subsequence, \(a(t_n)w_{\mathcal B}F^\circ\) converges in
\(G/F^\circ\).
Projecting to the coordinates in \(\mathcal J_{\mathrm{tor}}\), we obtain
that
\[
a_{\mathcal J_{\mathrm{tor}}}(t_n)w_{\mathcal B}
F_{\mathcal J_{\mathrm{tor}}}^\circ
\]
converges in
\(G_{\mathcal J_{\mathrm{tor}}}/F_{\mathcal J_{\mathrm{tor}}}^\circ\).
Since the noncompact part of
\(F_{\mathcal J_{\mathrm{tor}}}^\circ\) is \(S_{\mathrm{tor}}\), this is
possible only if
\[
w_{\mathcal B}^{-1}A_{\mathrm{tor}}w_{\mathcal B}
\subseteq S_{\mathrm{tor}}.
\]
Thus \(w_{\mathcal B}\in \mathcal{W}_{\mathrm{tor}}\), proving Claim~4.

We now prove item~(3). Let \(g'\in\Delta_H\), and put
\[
g=g'\xi^{-1}.
\]
Since the orbit $G(\xi p_H)$ is closed and $F$ is the stabilizer of
\(\xi p_H\), the map
\[
G/F\longrightarrow G(\xi p_H),
\qquad
hF\longmapsto h(\xi p_H),
\]
is a homeomorphism. Hence
\begin{equation}\label{eq:SH-convergence-GF}
a(t)gF
\end{equation}
converges in \(G/F\) as \(t\to\infty\). 

The Bruhat decomposition of \(G\) can be written as the disjoint union
of \(2^k\) cells:
\[
G
=
\bigsqcup_{\mathcal B\subseteq\{1,\dots,k\}}
P^-w_{\mathcal B}N_{\mathcal B^c}.
\]
Hence there exist
\(\mathcal B\subseteq\{1,\dots,k\}\), \(p_{g}\in P^-\),
and \(\mathbf y_i\in\mathbb R^{n_i-1}\) for
\(i\in\mathcal B^c\) such that
\begin{equation}\label{eq:SH-Bruhat}
g
=
p_{g}w_{\mathcal B}
\prod_{i\in\mathcal B^c}u_i(\mathbf y_i).
\end{equation}
Since \(a(t)p_{g}a(-t)\) converges in \(G\), it follows from
\eqref{eq:SH-convergence-GF} that
\begin{equation}\label{eq:SH-Bruhat-convergence}
a(t)w_{\mathcal B}
\prod_{i\in\mathcal B^c}u_i(\mathbf y_i)F
\end{equation}
converges in \(G/F\). We now show that

\begin{equation}\label{eq: unipotents get abosrbed}
	\prod_{i\in\mathcal B^c}u_i(\mathbf y_i) \in F^\circ.
\end{equation} 

To see this, we first claim that, for every \(\mathcal J=\mathcal J_j\in\mathscr P\),
\begin{equation}\label{eq:SH-no-partial-block}
\mathcal B\cap\mathcal J=\emptyset
\qquad\text{or}\qquad
\mathcal J\subseteq\mathcal B,
\end{equation} which will allow us to replace \(w_{\mathcal B}\) by an element of $\mathcal{W}_\mathrm{tor}$. This claim, combined with \eqref{eq: unipotents get abosrbed}, will yield $g \in P^-F^\circ \mathcal{W}_{\mathrm{tor}}$. 
Suppose, to the contrary, that
\[
\emptyset\neq
\mathcal J_{\mathrm{in}}:=\mathcal B\cap\mathcal J
\neq\mathcal J,
\qquad
\mathcal J_{\mathrm{out}}:=\mathcal J\setminus\mathcal J_{\mathrm{in}}.
\]
The projection of \eqref{eq:SH-Bruhat-convergence} to
\(G_{\mathcal J_{\mathrm{out}}}/F_{\mathcal J_{\mathrm{out}}}\) shows that
\begin{equation*}
a_{\mathcal J_{\mathrm{out}}}(t)
\prod_{i\in\mathcal J_{\mathrm{out}}}u_i(\mathbf y_i)F_{\mathcal J_{\mathrm{out}}}
\end{equation*}
converges. By Claim~3,
\begin{equation}\label{eq:SH-unipotent-in-projection}
\prod_{i\in\mathcal J_{\mathrm{out}}}u_i(\mathbf y_i)
\in F_{\mathcal J_{\mathrm{out}}}.
\end{equation}
Every unipotent element of \(F_{\mathcal J_{\mathrm{out}}}\) belongs to
\(F_{\mathcal J_{\mathrm{out}}}^\circ\). Since the compact normal factor contains no
nontrivial unipotent elements, the element in
\eqref{eq:SH-unipotent-in-projection} actually belongs to
\(\pi_{\mathcal J_{\mathrm{out}}}(S_j)\). Let
\((\theta_i)_{i\in\mathcal J}\) be the automorphisms defining the
generalized diagonal embedding \(S_j\). Then, for every
\(i_1,i_2\in\mathcal J_{\mathrm{out}}\),
\begin{equation}\label{eq: compatibility of automorphism}
	\bigl(\theta_{i_2}\circ\theta_{i_1}^{-1}\bigr)
	\bigl(u_{i_1}(\mathbf y_{i_1})\bigr)
	=u_{i_2}(\mathbf y_{i_2}).
\end{equation}
Since
\(A_{\mathcal J}\subseteq S_j\), we have
\[
\bigl(\theta_{i_2}\circ\theta_{i_1}^{-1}\bigr)(a_{i_1}(t))=a_{i_2}(t)
\]
for all \(i_1,i_2\in\mathcal J\). Hence
\(\theta_{i_2}\circ\theta_{i_1}^{-1}\) maps elements of \(N_{i_1}\) lying in the embedded
copy of \(\mathrm{SO}(m_{\mathcal J},1)\) to elements of \(N_{i_2}\), so we may choose
\(\mathbf y_i\in\mathbb R^{n_i-1}\),
\(i\in\mathcal J_{\mathrm{in}}\), such that the relation in
\eqref{eq: compatibility of automorphism} extends to all \(i_1,i_2\in\mathcal J\).
Thus
\[
\prod_{i\in\mathcal J}u_i(\mathbf y_i)
\in\pi_{\mathcal J}(S_j)\subseteq F_{\mathcal J}.
\]
Projecting \eqref{eq:SH-Bruhat-convergence} to
\(G_{\mathcal J}/F_{\mathcal J}\) and using the preceding inclusion,
we see that
\begin{align*}
&a_{\mathcal J}(t)w_{\mathcal J_{\mathrm{in}}}
\prod_{i\in\mathcal J_{\mathrm{out}}}u_i(\mathbf y_i)
F_{\mathcal J}\\
&\qquad =
a_{\mathcal J}(t)w_{\mathcal J_{\mathrm{in}}}
\prod_{i\in\mathcal J_{\mathrm{in}}}u_i(-\mathbf y_i)
F_{\mathcal J}
\end{align*}
converges as \(t\to\infty\).

For each \(i\in\mathcal J_{\mathrm{in}}\), we have
\(w_i u_i(-\mathbf y_i)w_i^{-1}\in N_i^-\). Hence
\[
a_{\mathcal J}(t)w_{\mathcal J_{\mathrm{in}}}
\Bigl(\prod_{i\in\mathcal J_{\mathrm{in}}}u_i(-\mathbf y_i)\Bigr)
w_{\mathcal J_{\mathrm{in}}}^{-1}a_{\mathcal J}(-t)
\longrightarrow e.
\]
It follows that
\(a_{\mathcal J}(t)w_{\mathcal J_{\mathrm{in}}}F_{\mathcal J}\)
converges in \(G_{\mathcal J}/F_{\mathcal J}\).

Since \(a_{\mathcal J}(-t)\in S_j\subseteq F_{\mathcal J}\),
\begin{align*}
a_{\mathcal J}(t)w_{\mathcal J_{\mathrm{in}}}F_{\mathcal J}
&=w_{\mathcal J_{\mathrm{in}}}
a_{\mathcal J_{\mathrm{in}}}(-t)
a_{\mathcal J_{\mathrm{out}}}(t)F_{\mathcal J}\\
&=w_{\mathcal J_{\mathrm{in}}}
a_{\mathcal J_{\mathrm{out}}}(2t)F_{\mathcal J}.
\end{align*}
Thus \(a_{\mathcal J_{\mathrm{out}}}(2t)F_{\mathcal J}\)
converges in \(G_{\mathcal J}/F_{\mathcal J}\).

Choose a sequence \(t_n\to\infty\). Since
\(F_{\mathcal J}/F_{\mathcal J}^{\circ}\) is finite, after passing
to a subsequence,
\(a_{\mathcal J_{\mathrm{out}}}(2t_n)F_{\mathcal J}^{\circ}\)
converges in \(G_{\mathcal J}/F_{\mathcal J}^{\circ}\).
After passing to a further subsequence, we may choose
\[
r_n=(r_{i,n})_{i\in\mathcal J}\in\pi_{\mathcal J}(S_j)
\]
such that \(a_{\mathcal J_{\mathrm{out}}}(2t_n)r_n\)
converges in \(G_{\mathcal J}\).

Fix \(i_1\in\mathcal J_{\mathrm{in}}\). The \(i_1\)-th coordinate
of this convergent sequence is \(r_{i_1,n}\), so
\(r_{i_1,n}\) converges. For every \(i_2\in\mathcal J\),
\[
r_{i_2,n}
=(\theta_{i_2}\circ\theta_{i_1}^{-1})(r_{i_1,n}),
\]
and hence \(r_n\) converges in \(G_{\mathcal J}\).
This contradicts the convergence of
\(a_{\mathcal J_{\mathrm{out}}}(2t_n)r_n\), since
\(a_{\mathcal J_{\mathrm{out}}}(2t_n)\) diverges in \(G_{\mathcal J}\).
This proves \eqref{eq:SH-no-partial-block}.

We now establish \eqref{eq: unipotents get abosrbed}.
For every \(1\leq j\leq p\) such that
\(\mathcal J_j\cap\mathcal B=\emptyset\), projecting
\eqref{eq:SH-Bruhat-convergence} to
\(G_{\mathcal J_j}/F_{\mathcal J_j}\) and applying Claim~3 gives
\[
\prod_{i\in\mathcal J_j}u_i(\mathbf y_i)
\in F_{\mathcal J_j}.
\]
Since this element is unipotent, it belongs to
\(S_j\).
On the toral coordinates, Claim~3 gives
\[
\prod_{i\in\mathcal J_{\mathrm{tor}}\cap\mathcal B^c}
u_i(\mathbf y_i)
\in F_{\mathcal J_{\mathrm{tor}}\cap\mathcal B^c}.
\]
Since $F_{\mathcal J_{\mathrm{tor}} \cap \mathcal{B}^c}$ contains no nontrivial unipotent element, we have
\[
\mathbf y_i=0
\qquad
(i\in\mathcal J_{\mathrm{tor}}\cap\mathcal B^c),
\]
and therefore
\[
\prod_{i\in\mathcal B^c}u_i(\mathbf y_i)
\in F^\circ.
\]

Set
\[
\mathcal B_{\mathrm{tor}}
:=
\mathcal B\cap\mathcal J_{\mathrm{tor}}.
\]
By \eqref{eq:SH-no-partial-block},
\(\mathcal B\setminus\mathcal B_{\mathrm{tor}}\) is a disjoint union of some $\mathcal{J}_j$. Thus \eqref{eq:SH-full-block-Weyl} and \eqref{eq:SH-Bruhat} give
\[
g=p_{g'}'w_{\mathcal B_{\mathrm{tor}}}f
\quad\text{for some }p_{g'}'\in P^-,\ f\in F^\circ.
\]
Since \(a(t)gF\) and \(a(t)p_{g'}'a(-t)\) converge,
\[
a(t)w_{\mathcal B_{\mathrm{tor}}}F
\]
converges in \(G/F\). Hence Claim~4 gives
\[
w_{\mathcal B_{\mathrm{tor}}}\in \mathcal{W}_{\mathrm{tor}}.
\]
It follows that
\[
g\in P^-\mathcal{W}_{\mathrm{tor}}F^{\circ} = P^-F^{\circ}\mathcal{W}_{\mathrm{tor}}.
\]
Equivalently,
\[
g' = g \xi \in P^-F^\circ \mathcal{W}_{\mathrm{tor}}\xi.
\]
This proves
\[
\Delta_H\subseteq P^-F^\circ \mathcal{W}_{\mathrm{tor}}\xi.
\]

For the converse inclusion, note that $F^\circ$ fixes \(\xi p_H\). Moreover, for each \(w\in \mathcal{W}_{\mathrm{tor}}\), its definition gives 
\[
w^{-1}Aw\subseteq F^\circ.
\]
Therefore every vector in \(\mathcal{W}_{\mathrm{tor}}F^\circ\xi p_H\)
is fixed by \(A\), so
\(\mathcal{W}_{\mathrm{tor}}F^\circ\xi p_H\subseteq V_L^{0-}(A)\). Since
\(P^-V_L^{0-}(A)\subseteq V_L^{0-}(A)\),
\eqref{eq:SH-move-W-right} and \eqref{eq:SH-disconnected-absorbed}
give the reverse inclusion and complete the proof.
\end{proof}

Let \(H\in\mathscr H_0\), and retain the notation of Proposition~\ref{prop:closed-orbit-obstruction}. Then,
\(\mathcal J_{\mathrm{tor}}=\emptyset\) and
\(\mathcal{W}_{\mathrm{tor}}=\{e\}\). Since \(S_0\subseteq P^-\),
Proposition~\ref{prop:closed-orbit-obstruction}(3) gives
\[
\Delta_H=P^-S_1\cdots S_p\xi.
\]
Recall that \(\mathcal I:G\to P^-\backslash G\) is the quotient map.
We explain why
\[
\mathcal I(\Delta_H)=\mathcal I(S_1\cdots S_p\xi)
\]
is the image of a product of diagonal M\"obius embeddings in the sense
of Definition~\ref{def: 1. obstruction set}.

Fix a block \(\mathcal J=\mathcal J_j\) and write
\(m=m_{\mathcal J}\). Choosing one coordinate \(i_0\in\mathcal J\) as the reference, we may express \(S_j\) as
\[
S_j=
\left\{
\bigl(\theta_i(h)\bigr)_{i\in\mathcal J}:
h\in\mathrm{SO}(m,1)
\right\},
\qquad
\theta_{i_0}=\mathrm{id},
\]where the \(\theta_i\) are automorphisms. Since \(A_{\mathcal J}\subseteq S_j\), the automorphisms
$\theta_i$ preserve the contracting parabolic subgroup \(P_m^-\) of $\mathrm{SO}(m,1)$ associated with the diagonal one-parameter subgroup. 
Hence the map
\[
\begin{aligned}
P_m^-\backslash\mathrm{SO}(m,1)
&\longrightarrow
\prod_{i\in\mathcal J}P_i^-\backslash G_i,\\
P_m^-h
&\longmapsto
\bigl(P_i^-\theta_i(h)\pi_i(\xi)\bigr)_{i\in\mathcal J}
\end{aligned}
\]
is well-defined.

Under the identifications
\[
P_m^-\backslash\mathrm{SO}(m,1)\cong\mathbb S^{m-1},
\qquad
P_i^-\backslash G_i\cong\mathbb S^{n_i-1},
\]
the standard group inclusion induces the standard sphere inclusion.
Geometrically, this is the boundary map of the standard totally
geodesic inclusion \(\mathbb H^m\hookrightarrow\mathbb H^{n_i}\);
compare~\cite[p.~284]{Sha09SOn1smooth}.

Each automorphism \(\theta_i\) is induced by conjugation by an
element of \(\mathrm O(m,1)\), and therefore
induces a M\"obius transformation on $\mathbb{S}^{m-1}$.
See~\cite[Sections~3.1, 3.3 and Theorem~3.7.7]{Beardon83}.
This M\"obius transformation can be extended to a M\"obius
transformation of the ambient sphere. Right translation by
\(\pi_i(\xi)\) also induces a M\"obius transformation on
\(P_i^-\backslash G_i\). Thus each coordinate of the displayed map
is the standard sphere inclusion followed by an ambient M\"obius
transformation, exactly as required in Section~\ref{intro.ch}.

Each block therefore gives a diagonal M\"obius embedding. Hence, $\mathcal{I}(\Delta_H)$ is the image of a product of diagonal M\"obius embeddings, as claimed.

%\int_X f\Big(u\Big(r\cdot\big(\varphi'_1(s_0), \cdots & , \varphi_{k_1}'(s_0),0, \cdots, 0\big)\Big)y\Big)d\mu_{s_0}(y) \\

\subsection{Case 2: \texorpdfstring{$Gp_H$}{GpH} is not closed}\label{sec:nonclosed-orbits} To identify the obstructions to equidistribution in this case, we need the following lemma and propositions. 

\begin{lem}[cf. \cite{BSX}, Lemma A.1 (Kempf)]\label{lemma: 2. convert not closed to unstable}
    Let \(G'\) be the group of \(\R\)-points of an algebraic group \(\mathbf G'\) over \(\R\).
    Let \(V\) be a rational representation of \(\mathbf G'\) and \(v \in V\).
    Let \(\mathcal O=\mathbf G'\cdot v\) be the algebraic orbit.
    Define \(S = (\text{Zcl}(\mathcal O)\backslash\mathcal O)(\R)\), where \(\text{Zcl}(\cdot)\) is the Zariski closure of a given subset in \(V\).
    Suppose that $S$ is nonempty. 
    Then, there exists a rational representation $W$ of $G'$ and a $G'$-equivariant polynomial map $P: V \to W$ such that $P(S) = \{0\}$ and $P(v) \neq 0$. 
\end{lem}

Let $V$ be a representation of $G$ and $v$ be a nonzero vector in $V$. We call $v$ unstable if $\text{Zcl}(G\cdot v)$ contains the origin. The following proposition is an adaptation of a result from \cite{SY24} for our purposes.

\begin{prop}[cf.~{\cite{SY24}, Proposition 2.2}]\label{proposition: computable lower bound}
    Let $G = \mathrm{SO}(n_1,1) \times  \cdots \times \mathrm{SO}(n_k,1)$, $V$ be a finite-dimensional representation of $G$ over $\R$ and $v$ be unstable in $V$.
    Then, there exist a nonzero tuple $(p_1, p_2, \dots, p_k) \in \Z_{\geq 0}^k$, an element $g_0 \in G$, and constants $C>0$ and $\beta>0$ such that the following holds: Let
        \[V' = (\R^{n_1+1})^{\otimes p_1}  \bigotimes \cdots \bigotimes (\R^{n_k+1})^{\otimes p_k}\]
        be the representation of $G$ such that $\mathrm{SO}(n_i,1)$ acts on $\R^{n_i+1}$ as the standard representation for each $1\leq i\leq k$, and
        let $w = e_0^{\otimes p_1} \otimes \cdots \otimes e_0^{\otimes p_k} \in V'$, where for each $1 \leq i \leq k$, $e_0 = (1, 0, \ldots, 0) \in \mathbb{R}^{n_i+1}$ denotes the first standard basis vector, and $e_0^{\otimes p_i} \in (\mathbb{R}^{n_i+1})^{\otimes p_i}$. Here, we use the convention $(\mathbb{R}^{n_i+1})^{\otimes 0}=\mathbb{R}$ and $e_0^{\otimes 0}=1$.
        Then, for any $g \in G$, 
            \begin{equation}\label{equation: Kempf computable lower}
                \norm{gg_0w}\leq C\norm{gv}^{\beta}.
            \end{equation}
    
\end{prop}

 Proposition 2.2 in \cite{SY24} was originally formulated for a connected $\mathbb{K}$-split semisimple group where $\mathbb{K}$ denotes a field of characteristic 0. The proposition established the existence of an irreducible representation $V'$ of $G$ and a highest weight vector $w$ in $V'$ that satisfy the inequality (\ref{equation: Kempf computable lower}) for some $g_0 \in G$, $C>0$ and $\beta>0$. However, it did not specify the particular representation $V'$ and the vector $w$.

In our situation, the group $G = \mathrm{SO}(n_1, 1) \times \cdots \times \mathrm{SO}(n_k,1)$ need not be $\R$-split. Nevertheless, the proof of Proposition 2.2 from \cite{SY24} remains valid, provided
we choose the specific $V'$ and $w$ as described in our Proposition \ref{proposition: computable lower bound}. This requires extending certain definitions in the context of connected $\mathbb{K}$-split semisimple groups to our group $G$ in an appropriate manner, as follows. 

 Recall that $G=\mathbf G(\R)^\circ$. All algebraic notions below are understood
 with respect to $\mathbf G$. Since $G$ is Zariski dense in $\mathbf G$,
 the relevant Zariski closures and the notion of instability are
 unchanged upon restriction from $\mathbf G(\R)$ to $G$.

 Let $S$ be the standard maximal $\R$-split torus of $G$.
 We denote the group of cocharacters of $S$ defined over $\R$ by $X_*(S)\coloneqq \text{Hom}(\mathbb{G}_m, S) \cong \Z^k$, and 
 the group of characters of $S$ defined over $\R$ by
 $X^*(S)\coloneqq \text{Hom}(S, \mathbb{G}_m) \cong \Z^k$.
 Then, any cocharacter $\delta \in X_*(S)$ can be expressed as  $\delta = (p_1, p_2, \dots, p_k)$ for some $(p_1, \dots, p_k) \in \Z^k$.
 Similarly, any character $\alpha \in X^*(S)$ can be expressed as $\alpha = (q_1, q_2, \dots , q_k)$ for some $(q_1, q_2,  \dots , q_k) \in \Z^k$.
 We have a pairing $\langle \cdot, \cdot\rangle:X^*(S) \times X_*(S) \to \mathrm{Hom}(\mathbb{G}_m, \mathbb{G}_m ) \cong \Z$ given by $\alpha \circ \delta(t) = t^{\langle \alpha, \delta \rangle}$ for any $t \in \C^*$.
 Let $\Phi^+(G,S)$ be the set of positive roots on $S$ for the adjoint action of $S$ on the Lie algebra $\mathfrak{g}$. We then define a $\Q$-valued positive definite bilinear form $(\cdot, \cdot)$ on $X_*(S)$ by
\begin{equation*}
    (\lambda, \lambda') = \sum_{\alpha \in \Phi^+(G,S)} \langle \alpha, \lambda \rangle \langle \alpha, \lambda'\rangle  = \sum_{i=1}^k p_i \cdot p_i'
\end{equation*} for all $\lambda = (p_1, \dots , p_k)$ and $\lambda' = (p_1', \dots, p_k') \in X_*(S)$ and extend it to $X_*(S)\tensor \Q$. Finally, define an injective $\Z$-module homomorphism $\, \hat{\cdot}\, :  X_*(S) \to X^*(S)$ by 
\begin{equation*}
    \hat{\delta} = (p_1, \dots, p_k)
\end{equation*} for each $\delta = (p_1, \dots, p_k) \in X_*(S)$.
As in the proof of \cite[Proposition~2.2]{SY24}, after conjugating a
Kempf optimal cocharacter for $v$ by a suitable $g_0\in G$ (rather
than merely by an element of $\mathbf G(\R)$), we obtain a dominant
cocharacter $\delta=(p_1,\dots,p_k)$, so that $p_i\geq 0$ for every
$i$. For this choice of $\delta$, the vector
\[
w=e_0^{\otimes p_1}\otimes\cdots\otimes e_0^{\otimes p_k}
\]
is a highest weight vector of $S$-weight $\hat{\delta}$. Since the
proof does not require $V'$ to be irreducible, as observed in
\cite[Remark~2.3]{SY24}, the same proof applies to this choice of
$V'$ and $w$.

\begin{prop}\label{prop: 7. unstable gamma H constant}
Let \(H\in\mathscr H\) be such that \(Gp_H\) is not closed, and let
\[
S_H=
\left\{
\mathbf x\in\prod_{i=1}^k\mathbb R^{n_i-1}:
u(\mathbf x)p_H\in V_L^{0-}(A)
\right\}.
\]
 Then, there exist constants $(C_{H})_j \in \R^{n_j-1}$ for $j = 1, \dots, k$ such that
\begin{equation*}
    S_H \subset \bigcup_{j = 1}^k \{ \mathbf{x} = (\mathbf{x}_1, \dots, \mathbf{x}_k) \in \prod_{i=1}^k \mathbb{R}^{n_i-1} : \mathbf{x}_j = (C_{H})_j\}.
\end{equation*}
\end{prop}

\begin{proof}
Recall that the $G$-action on $V_L$ extends uniquely to
a rational action of $\mathbf G$.
Let \(\mathcal O=\mathbf G\cdot p_H\) be the algebraic orbit.
Since every \(G\)-orbit in \(\mathcal O(\R)\) is open and closed,
every point of
\[
\overline{Gp_H}\setminus Gp_H,
\]
where the closure is taken in the usual topology, lies outside
\(\mathcal O(\R)\). Consequently,
\[
\overline{Gp_H}\setminus Gp_H
\subseteq
\bigl(\operatorname{Zcl}(\mathcal O)\setminus\mathcal O\bigr)(\R).
\]
Applying Lemma~\ref{lemma: 2. convert not closed to unstable} to the action
of \(\mathbf G(\R)\), we obtain an equivariant polynomial map \(P\) such
that \(P(p_H)\neq 0\) and
\[
P\bigl(\overline{Gp_H}\setminus Gp_H\bigr)=\{0\}.
\]
Thus \(P(p_H)\) is \(G\)-unstable. Therefore, it suffices to prove the
desired containment using the \(G\)-unstable vector \(P(p_H)\) in place of
\(p_H\). Hence, without loss of generality, we assume that \(p_H\) is
\(G\)-unstable.
Furthermore, in view of Proposition \ref{proposition: computable lower bound}, there exist $g_0 \in G$, a nonzero $(p_1, \dots, p_k)\in \Z_{\geq 0}^k$ and $C$, $\beta>0$ such that for a representation $V'$ and a vector $w \in V'$ given in Proposition \ref{proposition: computable lower bound}
and  for any $g \in G$, 
    \begin{equation}\label{equation: 6. upper bound}
        \norm{gg_0w}\leq C\norm{g p_H}^{\beta}
    \end{equation} 
    holds. Set $J_+=\{i:p_i>0\}$, which is nonempty. We can choose the sup-norm on $V'$ to satisfy the cross-norm property.
Fix a sequence $t_n\to\infty$, and let $\mathbf{x}\in S_H$. In view of the definition of $S_{H}$ and equation (\ref{equation: 6. upper bound}), the sequence $\{a(t_n)u(\mathbf{x})g_0 w\}_n$ is bounded. 
For each $i\in J_+$, put
\[
v_i=\pi_i(g_0)e_0=(v_{i0},v_{i1},\dots,v_{in_i})\in\R^{n_i+1}.
\]
Then, omitting the trivial tensor factors corresponding to $i\notin J_+$, we have
\[
g_0w=\bigotimes_{i\in J_+}v_i^{\otimes p_i}.
\]
Then, we have
\begin{equation}\label{equation: a(t)u(x)v i}
    \begin{split}
        a_i(t_n)&u_i(\mathbf{x}_i)v_i\\
        &= e^{\zeta_i t_n}(v_{i0} + \sum_{j = 1}^{n_i -1}v_{ij}\mathbf{x}_{ij} + v_{in_i}\frac{\norm{\mathbf{x}_i}_2^2}{2})\mathbf{e}_0 + \sum_{j=1}^{n_i-1}(v_{ij} + v_{in_i}\mathbf{x}_{ij})\mathbf{e}_j +v_{in_i}e^{-\zeta_i t_n}\mathbf{e}_{n_i}  
    \end{split}
\end{equation} for every $i\in J_+$, where $\{\mathbf{e}_0, \mathbf{e}_1, \dots, \mathbf{e}_{n_i}\}$ is the standard basis of $\R^{n_i+1}$ and $\mathbf{x}_i = (\mathbf{x}_{i1}, \dots, \mathbf{x}_{i,n_i-1}) \in \R^{n_i-1}$.

If $v_{i0} + \sum_{j = 1}^{n_i -1}v_{ij}\mathbf{x}_{ij} + v_{in_i}\frac{\norm{\mathbf{x}_i}_2^2}{2} \neq 0$ for all $i \in J_+$, then for large enough $n$,
\begin{equation*}
    \begin{split}
        \norm{a(t_n)u(\mathbf{x})g_0w} & =\prod_{i\in J_+}\norm{a_i(t_n)u_i(\mathbf{x}_i)v_i}^{p_i} \text{, by the cross-norm property}\\
        &\geq \prod_{i \in J_+} |e^{\zeta_i t_n}(v_{i0} + \sum_{j = 1}^{n_i -1}v_{ij}\mathbf{x}_{ij} + v_{in_i}\frac{\norm{\mathbf{x}_i}_2^2}{2})|^{p_i}
    \end{split}
\end{equation*}
and the right-hand side tends to $\infty$ as $n \to \infty$.
Therefore, we conclude that 
\begin{equation}\label{equation: 6 quadratic equation}
     v_{i0} + \sum_{j = 1}^{n_i -1}v_{ij}\mathbf{x}_{ij} + v_{in_i}\frac{\norm{\mathbf{x}_i}_2^2}{2} = 0
\end{equation} for some $i \in J_+$, since $a(t_n)u(\mathbf{x})g_0w$ is bounded.

Suppose that for a fixed $i\in J_+$, (\ref{equation: 6 quadratic equation}) holds.
Consider the quadratic form 
\begin{equation*}
    Q_{n_i}(x_0, x_1, \dots, x_{n_i}) =2x_0x_{n_i}-(x_1^2 + x_2^2 + \cdots + x_{n_i-1}^2).
\end{equation*} and note that $\mathbf{e}_0 \in \R^{n_i+1}$ is a solution to $Q_{n_i} = 0$. 
By the definition of $\SO(n_i,1)$,  \[a_i(t_n)u_i(\mathbf{x}_i)v_i = a_i(t_n)u_i(\mathbf{x}_i)\pi_i(g_0)e_0\] is still a solution to $Q_{n_i} = 0$. Furthermore, if a nonzero vector $(x_0,x_1,\dots,x_{n_i})$ satisfies $Q_{n_i}(x_0,\dots,x_{n_i})=0$ and $x_0=0$, then $x_1=x_2=\cdots=x_{n_i-1}=0$ and $x_{n_i}\neq0$. Therefore, by \eqref{equation: a(t)u(x)v i} and \eqref{equation: 6 quadratic equation}, we have
\begin{equation*}
    v_{i1} + v_{in_i}\mathbf{x}_{i1} = v_{i2} + v_{in_i}\mathbf{x}_{i2} = \cdots = 
    v_{i(n_{i}-1)} + v_{in_i}\mathbf{x}_{i,n_i-1} = 0 
\end{equation*} 
and $v_{in_i} \neq 0$. 
Then, 
\begin{equation*}
    \mathbf{x}_i = (-\frac{v_{i1}}{v_{in_i}}, \dots , -\frac{v_{i(n_i-1)}}{v_{in_i}})
\end{equation*} holds. For every $i\in J_+$ with $v_{in_i}\neq 0$, define
\[
    (C_H)_i=(-\frac{v_{i1}}{v_{in_i}}, \dots , -\frac{v_{i(n_i-1)}}{v_{in_i}}),
\]
and choose $(C_H)_i\in\R^{n_i-1}$ arbitrarily for the remaining indices. The preceding argument shows that
\[
S_H\subset \bigcup_{i\in J_+}\{\mathbf{x}:\mathbf{x}_i=(C_H)_i\}
\subset \bigcup_{i=1}^k\{\mathbf{x}:\mathbf{x}_i=(C_H)_i\}.
\]
This completes the proof.
\end{proof}

\begin{cor}\label{cor:nonclosed-orbit-nullity}
Under the assumptions of Theorem~\ref{Theorem : 1. main theorem},
the set \(I_H\) is null for every \(H\in\mathscr H_2\).
Consequently, \(E_2\) is a null set.
\end{cor}

\begin{proof}
Recall the level-set observation in
the proof of Corollary~\ref{cor:closed-orbit-nullity}. By
Proposition~\ref{prop: 7. unstable gamma H constant}, for every
\(H\in\mathscr H\) such that \(Gp_H\) is not closed,
\[
I_H
\subseteq
\bigcup_{j=1}^k
\{s\in I:\varphi_j(s)=(C_H)_j\}.
\]
Every set on the right-hand side is null. Thus \(I_H\) is null
whenever \(Gp_H\) is not closed. Since \(\mathscr H\) is countable,
it follows that
\[
E_2
=
\bigcup_{\substack{H\in\mathscr H\\Gp_H\text{ is not closed}}}I_H
\]
is a null set.
\end{proof}

\subsection{Proof of Theorem \ref{thm: 2. equidistribution of polyn approx}}\label{sec:proof-polynomial-equidistribution}
Recall that \(E_1\) is null by Corollary~\ref{cor:closed-orbit-nullity},
\(E_2\) is null by Corollary~\ref{cor:nonclosed-orbit-nullity}, and \(E_3\) is null by the
nonvanishing derivative assumption in Theorem~\ref{Theorem : 1. main theorem}.
Hence \(E=E_1\cup E_2\cup E_3\) is a null set. We retain the choice
\(s_0\in I\setminus E\) made in Section~\ref{Section: Shrinking}.
We continue to work with the sequence \(t_n\to\infty\) along which \(\mu_{s_0,t_n}\to\mu_{s_0}\), and write \(s_n=s_{t_n}\) and \(R(h)=R_{s_n}(h)\).

In Section \ref{Sec: nondiv and unip inv}, we showed the non-divergence of $\mu_{s_0,t}$ and the unipotent invariance of $\mu_{s_0}$. Now, it remains to show that $\mu_{s_0}$ is, indeed, $\mu_L$. 

Let $W$ be the largest connected unipotent subgroup of $N$ such that $\mu_{s_0}$ is invariant under the action of $W$. Then by Proposition~\ref{prop: unip inv}, we have $\{u(r\mathbf{w}_{k_1}(s_0)): r \in \R\} \subset W$, hence $\dim{W}\geq 1$.

Recall that 
\begin{equation*}
\begin{split}
    \mathscr{H} = \{H \lneq L : \,  H \text{ is clo} &\text{sed and connected, } H \cap \Gamma \text{ is a lattice in } H, \text{ and some nontrivial }\\ 
    &\Ad_L\text{-unipotent } 
    \text{one-parameter subgroup in } H \text{ acts ergodically on } H/(H\cap \Gamma). \}.
\end{split}
\end{equation*} 
For $H \in \mathscr{H}$, define 
\begin{equation*}
    \begin{split}
        & N(W,H) = \{g \in L : g^{-1}Wg \subset H\} \quad \text{ and } \\
        &S(W,H) = \bigcup_{F \lneq H, F \in \mathscr{H} }N(W,F).
    \end{split}
\end{equation*}
Then, we have
\begin{equation*}
    N(W, H)N_L(H) = N(W,H)
\end{equation*} and
\begin{equation*}
    \big(N(W,H)\backslash S(W,H)\big)x_0 = \big(N(W,H)x_0\big)\backslash\big(S(W,H)x_0\big)
\end{equation*} in $X$.

For the sake of contradiction, suppose that $\mu_{s_0}\neq\mu_L$.
By Ratner's theorem~\cite[Theorem 1]{Rat91}, we can choose $H \in \mathscr{H} $ such that 
\begin{equation} \label{eq:N-S}
\mu_{s_0}\bigl(N(W,H)x_0\bigr)>0\quad\text{and}\quad\mu_{s_0}\bigl(S(W,H)x_0\bigr)=0.
\end{equation}

By Theorem 2.2 from \cite{MozesShah95}, {\it every $W$-ergodic component of $\mu_{s_0}|_{(N(W,H)\backslash S(W,H))x_0}$ is the unique $gHg^{-1}$-invariant probability measure $g\mu_H$ on $gHx_0$ for some $g \in N(W,H)$ where $\mu_H$ is the $H$-invariant probability measure on $Hx_0$.}

\begin{prop}\label{prop: H not normal}
    $H$ is not a normal subgroup of $L$.
\end{prop}

\begin{proof} Suppose that $H$ is a normal subgroup of $L$. Since $\mu_{s_0}\big(N(W,H)x_0)>0$, $N(W,H)$ is nonempty. Therefore, $W \subset H$ and $N(W,H) = L$. 
Now, let $\bar L = L/H$, $q: L \to \bar L$ be the quotient homomorphism, $\bar \Gamma = q(\Gamma)$, $\bar G = q(G) = \prod_{i \, : \, G_i \not \subset H} G_i$. Since $H\Gamma$ is closed in $L$, $\bar X : = \bar L/\bar \Gamma$ is a finite-volume homogeneous space. Let \(\bar q:X\to\bar X\) be the induced quotient map. For any $x \in X$, we denote $\bar x = \bar q(x)$.  Let $\mathcal{M}^1(X)$ and $\mathcal{M}^1(\bar X)$ be the spaces of Borel probability measures on $X$ and $\bar X$, respectively. Define $\bar q_{\ast}: \mathcal{M}^1(X) \to \mathcal{M}^1(\bar X)$ by $\bar q_{\ast}(\lambda)(B) = \lambda(\bar q^{-1}(B))$ for each $\lambda \in \mathcal{M}^1(X)$ and every Borel set $B$ in $\bar X$. Then $\bar q_{\ast}$ is continuous. For any $\lambda \in \mathcal{M}^1(X)$, let $\bar q_{\ast}(\lambda) = \bar \lambda$. 

Set \[\overline{\mu}_{s_0,t_n}=\bar q_*\mu_{s_0,t_n}.\] These are the parameter measures of the projected curves
\(\eta\mapsto\bar q\bigl(a(t_n)u(R_{s_n}(\eta e^{-t_n/2}))u(\varphi(s_n))x_0\bigr)\), \(\eta\in I\). Since $\lim_{n \to \infty}\mu_{s_0, t_n} = \mu_{s_0}$ and $\bar q_{\ast}$ is continuous, we have \[ \lim _{n \to \infty} \overline{\mu_{s_0,t_n}} = \lim_{ n \to \infty} \bar q_{\ast}(\mu_{s_0,t_n}) = \bar q_{\ast}(\mu_{s_0}) = \overline{\mu_{s_0}}. \] 

Let $\mathcal{J} = \{1 \leq i \leq k : G_i \not\subset H\} = \{ i_1, i_2, \dots, i_p\}$ for some $p \in \N$, with $i_1<\cdots<i_p$. Let $p_1 $ be the smallest index $j \in \{1, 2, \dots, p-1\}$ such that $\zeta_{i_j}>\zeta_{i_{j+1}}$. If no such index exists (i.e., if $\zeta_{i_1} = \zeta_{i_2} = \cdots = \zeta_{i_p}$), we set $p_1 = p$. By the same argument as in Proposition~\ref{prop: unip inv}, $\overline{\mu_{s_0}}$ is invariant under $\{\big(u_{i_j}(r\cdot\varphi_{i_j}'(s_0))\big)_{1 \leq j \leq p_1}: r \in \R\}$.
 
For any $g\in L$ and $x=g\Gamma$, $\bar q^{-1}(\bar x)=gH\Gamma/\Gamma$.
Since $H$ is normal in $L$, $g\mu_H$ is the unique $H$-invariant
probability measure on this fiber and hence depends only on the coset
$gH\Gamma$. In view of the observation after \eqref{eq:N-S}, a
$W$-ergodic decomposition of $\mu_{s_0}$ can be expressed as:
\begin{equation}  \label{eq:ergodic-mus0}
    \mu_{s_0} = \int_{g \in \overline{\mathcal{F}}} g\mu_H \, d\overline{\mu_{s_0}}(gH\Gamma),
\end{equation}
where $\overline{\mathcal{F}}\subset L$ is any measurable set of
representatives for $\bar X=L/H\Gamma$.

We claim that $\mu_{s_0}$ is invariant under $\{\big(u_{i_j}(r\cdot\varphi_{i_j}'(s_0))\big)_{1 \leq j \leq p_1}: r \in \R\}$. 

To prove this claim, let $u\in \{\big(u_{i_j}(r\cdot\varphi_{i_j}'(s_0))\big)_{1 \leq j \leq p_1}: r \in \R\}$. Then, by \eqref{eq:ergodic-mus0},
\begin{align*}
    u\mu_{s_0} 
    &= \int_{g \in \overline{\mathcal{F}}} ug\mu_H \, d\overline{\mu_{s_0}}(gH\Gamma)\\
    &= \int_{g\in u\overline{\mathcal{F}}} g\mu_H \, d\overline{\mu_{s_0}}(gH\Gamma)\\
   &=\mu_{s_0},
\end{align*}
because $\overline{\mu_{s_0}}$ is $u$-invariant, and
$u\overline{\mathcal{F}}\subset L$ is a fundamental domain of $L/H\Gamma$.
This proves the claim.

However, since $\big(u_{i_j}(r \cdot \varphi_{i_j}'(s_0))\big)_{1 \leq j \leq p_1} \notin W \subset H$ for all $r \in \R \setminus \{0\}$, this is a contradiction to the maximality of $W$. Therefore, $H$ cannot be a normal subgroup of $L$. 
\end{proof}

Now, let $C$ be a compact subset of $N(W,H)\backslash S(W,H)$ such that $\mu_{s_0}(C\Gamma/\Gamma) \coloneqq \varepsilon_0$ for some $\varepsilon_0>0$. 
Define 
\begin{equation*}
    \mathcal{A} = \{v \in \wedge^{\text{dim} H}\mathfrak{L}: v \wedge X = 0 \text{ in } V_L \text{ for all } X \in \text{Lie}(W)\}.
\end{equation*} 
Then, 
\begin{equation*}
    \{g \in L : g\cdot p_H \in\mathcal{A}\} =  \{g \in L : \text{Lie}(W) \subset \Ad(g)(\text{Lie}(H))\} = N(W,H). 
\end{equation*}

Now we apply the linearization technique. The following theorem is an adaptation of Theorem 4.1 from \cite{Sha96} for our purposes.

\begin{thm}[\cite{Sha96}, Theorem 4.1]\label{Theorem: 7. linearization for equidistribution} 
    Let $\varepsilon>0$, $d \in \N$, and a compact set $\mathcal{C}$ in $(N(W,H)\backslash S(W,H)\big)\Gamma/\Gamma$ be given. Then, there exists a compact subset $\mathcal{D}$ of $\mathcal{A}$ such that for any open neighborhood $\Phi$ of $\mathcal{D}$ in $V_L$, there exists an open neighborhood $\Psi$ of $\mathcal{C}$ in $L/\Gamma$ such that for any $p \in \mathcal{P}_d(L)$ and for any bounded open interval $J$, one of the following holds:
    \begin{enumerate}
        \item there exists $\gamma \in \Gamma$ such that \[p(J)\gamma p_H \subset \Phi.\]
        \item \[\frac{1}{|J|}\left| \{s \in J: p(s)x_0 \in \Psi\}\right| < \varepsilon.\]
    \end{enumerate}
\end{thm}

\medskip

Now, since the sequence of polynomials $\{ \eta \mapsto a(t_n) u(R( e^{-\frac{1}{2}t_n} \eta))u(\varphi(s_n))\}_n$ has uniformly bounded degree, we can let $d$ be the maximum degree of the sequence.  
For $\frac{\varepsilon_0}{2}$, $d$ and $C\Gamma/\Gamma$, there exists a compact subset $\mathcal{D}$ of $\mathcal{A}$ given in Theorem \ref{Theorem: 7. linearization for equidistribution}.  Let $\Phi_1$ be a relatively compact open neighborhood of $\mathcal{D}$ in $V_L$ and $\Psi_1$ be the corresponding neighborhood of $C\Gamma/\Gamma$ in $L/\Gamma$.
Let 
\begin{equation*}
    I(\Psi_1,n) = \{\eta \in I: a(t_n)u(R( e^{-\frac{1}{2}t_n}\eta))u(\varphi(s_n))x_0 \in \Psi_1\}.
\end{equation*}
Then,
\begin{equation}\label{equation: 7. liminf of measure of open set}
    \varepsilon_0 = \mu_{s_0}(C\Gamma/\Gamma) \leq \mu_{s_0}(\Psi_1) \leq \liminf_n\mu_{s_0,t_n}(\Psi_1) = \liminf_n  \left|I(\Psi_1, n)\right|.
\end{equation}

If condition (2) of Theorem \ref{Theorem: 7. linearization for equidistribution} holds for $a(t_n) u(R( e^{-\frac{1}{2}t_n} \eta))u(\varphi(s_n))$  for infinitely many $n$, i.e., if 
\begin{equation*}
    \left|\{ \eta \in I : a(t_n) u(R( e^{-\frac{1}{2}t_n} \eta))u(\varphi(s_n))x_0 \in \Psi_1 \} \right| <\frac{\varepsilon_0}{2}
\end{equation*} for infinitely many $n$, then this contradicts (\ref{equation: 7. liminf of measure of open set}). Therefore, for all but finitely many $n$, there exists $\gamma_n \in \Gamma$ such that \begin{equation*}
    a(t_n)u(R(e^{-\frac{1}{2}t_n}I))u(\varphi(s_n))\gamma_np_H \subset \Phi_1.
\end{equation*}
Since $\Phi_1$ is relatively compact, we have 
\begin{equation*}
    \sup_{\eta \in I}\norm{a(t_n)u(R(e^{-\frac{1}{2}t_n}\eta))u(\varphi(s_n))\gamma_np_H}<M 
\end{equation*}
for some $M>0$. 
Proposition~\ref{prop: basic} and the convergence $s_n\to s_0$
give a uniform bound for $\gamma_n p_H$.
Since $\Gamma p_H$ is discrete, after passing to a subsequence,
there exists $\gamma \in \Gamma$ such that
$\gamma p_H = \gamma_n p_H$ for all $n$. Now we have 
\begin{equation*}
    \sup_{\eta \in I}\norm{a(t_n)u(R(e^{-\frac{1}{2}t_n}\eta))u(\varphi(s_n))\gamma p_H}<M.
\end{equation*}
Since $\gamma H\gamma^{-1} \in \mathscr{H}$, we can replace $\gamma p_H = p_{\gamma H \gamma^{-1}}$ with $p_H$.
Letting $\eta = 0$, we have 
\begin{equation}\label{equation: 4. u(phi(sn)) bounded}
    \norm{a(t_n)u(\varphi(s_n))p_H}<M. 
\end{equation} 

Now we claim that for any sequence $\{w_n\}_n$ in a finite-dimensional representation $V$ of $G$ such that $\lim_{n \to \infty}w_n = w$ for some $w \in V$, if 
$ \norm{a(t_n)w_n}<M$ for some $M>0$ and for all $n$, then $w \in V^{0-}(A).$ 

Let $V^+(A) = \{ v \in V : \lim_{t \to \infty} a(-t)v = 0\}$. Let $\text{Pr}^+: V \to V^+(A)$ be the projection onto $V^+(A)$. For the sake of contradiction, suppose that $\norm{\text{Pr}^+(w)} = c$ for some $c>0$. Then $\norm{\text{Pr}^+(w_n)} > \frac{c}{2}$ for large enough $n$. This implies that if $n$ is large enough, we have
\begin{equation*}
    \begin{split}
    \norm{a(t_n)w_n} & \geq \norm{a(t_n)\text{Pr}^+(w_n)}\\
    & \geq e^{C_1t_n}\norm{\text{Pr}^+(w_n)}\\
    & > e^{C_1 t_n}\cdot \frac{c}{2}
    \end{split}         
\end{equation*} for the smallest positive eigenvalue $C_1>0$ of $V$ with respect to $H_\zeta$. This contradicts $\norm{a(t_n)w_n}<M$ for all $n$. Therefore, $\text{Pr}^+(w) = 0$.

By this claim, (\ref{equation: 4. u(phi(sn)) bounded}) implies that
\begin{equation}\label{equation: upH is in V0-}
    u(\varphi(s_0))p_H \in V_L^{0-}(A).
\end{equation}

For this $p_H$, we further claim that $G$ does not fix $p_H$. If $Gp_H = \{p_H\}$, it would imply $G \subset N^1_L(H)$ where $N_L^1(H) = \{g \in N_L(H): \det(\Ad\, g|_{\text{Lie}(H)}) = 1  \}$.  Note that since $\Gamma p_H$ is closed, it follows that $\Gamma N^1_L(H) = \{g \in L: gp_H \in \Gamma p_H\}$ is closed. This implies that $N^1_L(H)\Gamma = (\Gamma N_L^1(H))^{-1}$ is closed as well. Consider $L = \overline{G\Gamma} \subset N^1_L(H)\Gamma$, which would imply $L = N^1_L(H)$, i.e., $H \triangleleft L$. This contradicts Proposition~\ref{prop: H not normal}.

Therefore, \eqref{equation: upH is in V0-} implies $s_0 \in E_1\cup E_2$, recalling that $E_1$ and $E_2$ are the exceptional sets corresponding to the closed and non-closed orbit cases, respectively (see \eqref{2. E definitions}). This contradicts the choice of $s_0$. Therefore, we conclude that $\mu_{s_0}$ is the unique $L$-invariant probability measure on $L/\Gamma$. This completes the proof of Theorem~\ref{thm: 2. equidistribution of polyn approx}. \qed

\subsection{Proof of Theorem \ref{Thm : 1. shrinking equidistribution at point}}
%This proof follows the proof of Theorem 1.3 from \cite{ShahPYang24}.

Recall that, by Taylor's theorem, for all $s,s+h\in I$, we may write
\begin{equation*}
    \varphi(s+h)=\varphi(s)+R_s(h)+o(h^l)
\end{equation*}
as $h\to0$. Here, the $o(h^l)$-term is uniform over all $s\in I$.

Now let $f\in C_c(X)$, $s_0\in I\setminus E$, $s_n\to s_0$, and $t_n\to\infty$, and put
\[
    h=\eta e^{-t_n/2}.
\]

Let $\varepsilon>0$ be given. Since $f$ is uniformly continuous, we can choose $\delta>0$ such that, for any $y,z\in X$, if $y\stackrel{\delta}{\approx}z$, then $f(y)\stackrel{\varepsilon/2}{\approx}f(z)$.

We have
\[
    a(t_n)u(o(h^l))a(t_n)^{-1}=u(o_n(1)),
\]
where $o_n(1)\to0$ uniformly for $\eta\in[0,1]$, since
\[
    \zeta_i-\frac{l}{2}\leq \zeta_i-\zeta_1\leq0
    \qquad\text{for every }i.
\]
Therefore,
\begin{equation*}
    \begin{split}
        &\int_0^1 f\big(a(t_n)u(\varphi(s_n+\eta e^{-t_n/2}))x_0\big)\,d\eta\\
        &\quad=\int_0^1 f\big(a(t_n)u(o(h^l))a(t_n)^{-1}a(t_n)u(\varphi(s_n)+R_{s_n}(h))x_0\big)\,d\eta\\
        &\quad=\int_0^1 f\big(u(o_n(1))a(t_n)u(\varphi(s_n)+R_{s_n}(h))x_0\big)\,d\eta\\
        &\quad\stackrel{\varepsilon/2}{\approx}\int_0^1 f\big(a(t_n)u(\varphi(s_n)+R_{s_n}(h))x_0\big)\,d\eta
        \qquad\text{for sufficiently large }n\\
        &\quad\stackrel{\varepsilon/2}{\approx}\int_X f\,d\mu_L
        \qquad\text{for sufficiently large }n,
    \end{split}
\end{equation*}
where the last approximation follows from Theorem~\ref{thm: 2. equidistribution of polyn approx}. Hence, for sufficiently large $n$,
\[
    \left|
    \int_0^1 f\big(a(t_n)u(\varphi(s_n+\eta e^{-t_n/2}))x_0\big)\,d\eta
    -\int_X f\,d\mu_L
    \right|<\varepsilon.
\]
Since $\varepsilon$ is arbitrary, this completes the proof.
\qed

\subsection{Proof of Theorem \ref{Theorem : 1. main theorem}}
This proof follows the proof of Theorem 1.3
from \cite{ShahPYang24}.
Note that $E$ is a Lebesgue null set. By equation (\ref{3. equation: measure nondivergent}) and Theorem \ref{Thm : 1. shrinking equidistribution at point}, we can derive that for any bounded continuous function $f \in C_b(X)$, the equation (\ref{equation: 1. equidist on shrinking curve}) still holds. It is enough to show that for any $f \in C_b(X)$ such that $\norm{f}_{\infty} \leq 1$ and $\int_Xf \, d\mu_L =0$, and for any compact set $K$ in $I\backslash E$ with $|K|>0$, 
\begin{equation}\label{equation: 5. equidist on cpt set}
    \lim_{t \to \infty} \frac{1}{|K|}\int_K f\big(a(t)u(\varphi(s))x_0\big)ds = 0.
\end{equation}

Suppose that (\ref{equation: 5. equidist on cpt set}) fails to hold for some $f$ and $K$. 
Then, there exists $\varepsilon>0$ and a sequence $\{t_n\}$ in $\R$ with $\lim_{ n \to \infty} t_n = \infty$ such that 
\begin{equation*}
    \Big|\int_K f(a(t_n) u( \varphi(s))x_0)ds\Big| > |K|\cdot \varepsilon
\end{equation*} holds for all $n$.

For each large $n$, we can choose finitely many disjoint intervals of the form $(s, s+e^{-\frac{1}{2}t_n})$ such that each interval has a nonempty intersection with $K$ and the symmetric difference between their union and $K$ has Lebesgue measure less than $\frac{|K|\varepsilon}{4}$.

Then, for each sufficiently large $n$, since $\norm{f}_{\infty}\leq 1$, we can choose $s_n$ such that $K \cap (s_n, s_n + e^{-\frac{1}{2}t_n}) \neq \emptyset$ and we have 
\begin{equation*}
    \Big|\int_{s_n}^{s_n + e^{-\frac{1}{2}t_n}} f(a(t_n)u(\varphi(s))x_0)ds\Big| >\frac{\varepsilon e^{-\frac{1}{2}t_n}}{2}.
\end{equation*}
This can be written as
\begin{equation*}
    \Big| \int ^1_0 f(a(t_n)u(\varphi(s_n + e^{-\frac{1}{2}t_n}\eta))x_0)d\eta\Big|> \frac{\varepsilon}{2}.
\end{equation*}
Note that by passing to a subsequence, $s_n \to s_0$ as $n \to \infty$ for some $s_0 \in K$. By Theorem \ref{Thm : 1. shrinking equidistribution at point}, this is a contradiction to our assumption $\int_X fd\mu_L = 0$. This completes the proof. \qed 

%\section{Geometric interpretation of countable obstructions}\label{sec: geometric meaning}

%
% If you have appendices in your dissertation, you will need the
% following, else keep it commented. The following appendices are in
% files called ``app1.tex'', and ``app2.tex'', and they
% look just like any chapter.
%

%\appendix
%\include{app1}
%\include{app2}

%
% The all important bibliography file at the end of your document!! Use
% the bibstyle you (your department) like in the \bibliographystyle{}
% statement and list the name of your bibliography database file in
% the \bibliography{} statement.  In this example, ``bibfile.bib'' is
% the name of the database.  See the LaTeX manual appendix B for details
% about the bibliography database and BibTeX.
%

\section*{Acknowledgements}
This article is based on the author's doctoral dissertation~\cite{Shin26thesis}.
The author is grateful to their advisor, Nimish Shah, for suggesting the problem and for valuable guidance throughout the research.

The main mathematical ideas of this paper were developed without AI assistance and finalized in April 2026.
In preparing the manuscript, the author used AI tools
to check proofs, simplify selected arguments, and assist with
editing. The author reviewed the resulting revisions and assumes
full responsibility for the paper.

\bibliographystyle{abbrv}
\bibliography{references}

@article{Aka18,
  title={{Diophantine} approximation on matrices and {Lie} groups},
  author={Aka, Menny and Breuillard, Emmanuel and Rosenzweig, Lior and de Saxc{\'e}, Nicolas},
  journal={Geometric and Functional Analysis},
  volume={28},
  number={1},
  pages={1--57},
  year={2018},
  publisher={Springer}
}

@book{Beardon83,
  author={Beardon, Alan F.},
  title={The Geometry of Discrete Groups},
  series={Graduate Texts in Mathematics},
  volume={91},
  publisher={Springer-Verlag},
  address={New York},
  year={1983},
  doi={10.1007/978-1-4612-1146-4}
}

@article{BCLS15,
  author={Bridgeman, Martin and Canary, Richard and Labourie, Fran{\c{c}}ois and Sambarino, Andr{\'e}s},
  title={The pressure metric for {Anosov} representations},
  journal={Geometric and Functional Analysis},
  volume={25},
  pages={1089--1179},
  year={2015},
  doi={10.1007/s00039-015-0333-8}
}

@article{BSX,
  title={Equidistribution of polynomially bounded o-minimal curves in homogeneous spaces},
  author={Bersudsky, Michael and Shah, Nimish A and Xing, Hao},
  journal={Duke Mathematical Journal},
  volume={175},
  number={3},
  pages={437--510},
  year={2026},
  publisher={Duke University Press},
  doi={10.1215/00127094-2025-0033}
}

@article{dani1984orbits,
  title={On orbits of unipotent flows on homogeneous spaces},
  author={Dani, S G},
  journal={Ergodic Theory and Dynamical Systems},
  volume={4},
  number={1},
  pages={25--34},
  year={1984},
  publisher={Cambridge University Press}
}

@incollection{DM93,
  author={Dani, S. G. and Margulis, G. A.},
  title={Limit distributions of orbits of unipotent flows and values of quadratic forms},
  booktitle={I. M. Gel'fand Seminar, Part 1},
  series={Advances in Soviet Mathematics},
  volume={16},
  pages={91--137},
  publisher={American Mathematical Society},
  address={Providence, RI},
  year={1993}
}

@article{Kempf78,
  title={Instability in invariant theory},
  author={Kempf, George R},
  journal={Annals of Mathematics},
  volume={108},
  number={2},
  pages={299--316},
  year={1978},
  publisher={JSTOR}
}

@article{KleinbockWeiss08,
  title={{Dirichlet}'s theorem on {Diophantine} approximation and homogeneous flows},
  author={Kleinbock, Dmitry and Weiss, Barak},
  journal={Journal of Modern Dynamics},
  volume={2},
  number={1},
  pages={43--62},
  year={2008}
}

@article{KM98,
  title={Flows on homogeneous spaces and {Diophantine} approximation on manifolds},
  author={Kleinbock, Dmitry and Margulis, Grigorij A},
  journal={Annals of Mathematics},
  volume={148},
  number={1},
  pages={339--360},
  year={1998},
  publisher={JSTOR}
}

@article{MozesShah95,
  title={On the space of ergodic invariant measures of unipotent flows},
  author={Mozes, Shahar and Shah, Nimish A},
  journal={Ergodic Theory and Dynamical Systems},
  volume={15},
  number={1},
  pages={149--159},
  year={1995},
  publisher={Cambridge University Press}
}

@article{Rat91,
  title={On {Raghunathan}'s measure conjecture},
  author={Ratner, Marina},
  journal={Annals of Mathematics},
  volume={134},
  number={3},
  pages={545--607},
  year={1991},
  publisher={JSTOR}
}

@article{Sha09SLnR,
  title={Equidistribution of expanding translates of curves and {Dirichlet}'s theorem on {Diophantine} approximation},
  author={Shah, Nimish A},
  journal={Inventiones Mathematicae},
  volume={177},
  pages={509--532},
  year={2009},
  publisher={Springer}
}

@article{Sha09SOn1analytic,
  title={Limiting distributions of curves under geodesic flow on hyperbolic manifolds},
  author={Shah, Nimish A},
  journal={Duke Mathematical Journal},
  volume={148},
  number={2},
  pages={251--279},
  year={2009},
  doi={10.1215/00127094-2009-026}
}

@article{Sha09SOn1smooth,
  title={Asymptotic evolution of smooth curves under geodesic flow on hyperbolic manifolds},
  author={Shah, Nimish A},
  journal={Duke Mathematical Journal},
  volume={148},
  number={2},
  pages={281--304},
  year={2009},
  doi={10.1215/00127094-2009-027}
}

@article{SY24,
  author={Shah, Nimish A and Yang, Pengyu},
  title={Equidistribution of expanding degenerate manifolds in the space of lattices},
  journal={Proceedings of the London Mathematical Society},
  volume={129},
  number={4},
  pages={e12634},
  year={2024},
  doi={10.1112/plms.12634}
}

@article{ShahLYang,
  title={Equidistribution of curves in homogeneous spaces and {Dirichlet}'s approximation theorem for matrices},
  author={Shah, Nimish A and Yang, Lei},
  journal={Discrete and Continuous Dynamical Systems},
  volume={40},
  number={9},
  pages={5247--5287},
  year={2020},
  publisher={Discrete and Continuous Dynamical Systems}}

@article{ShahPYang24,
  title={Equidistribution of non-uniformly stretching translates of shrinking smooth curves and weighted {Dirichlet} approximation},
  author={Shah, Nimish A and Yang, Pengyu},
  journal={Mathematische Zeitschrift},
  volume={308},
  number={4},
  pages={60},
  year={2024},
  publisher={Springer}
}

@article{Sha96,
  title={Limit distributions of expanding translates of certain orbits on homogeneous spaces},
  author={Shah, Nimish A},
  journal={Proceedings of the Indian Academy of Sciences - Mathematical Sciences},
  volume={106},
  number={2},
  pages={105--125},
  year={1996},
  publisher={Springer}
}

@phdthesis{Shin26thesis,
  author={Shin, Yubin},
  title={Equidistribution of expanding translates of smooth curves in homogeneous spaces under the action of a product of {$\mathrm{SO}(n,1)$}'s},
  school={The Ohio State University},
  year={2026}
}

@article{PYang20thesis,
  title={Equidistribution of expanding translates of curves and {Diophantine} approximation on matrices},
  author={Yang, Pengyu},
  journal={Inventiones Mathematicae},
  volume={220},
  number={3},
  pages={909--948},
  year={2020},
  publisher={Springer}
}

@article{LyangSquareMatrices,
  title={Equidistribution of expanding curves in homogeneous spaces and {Diophantine} approximation on square matrices},
  author={Yang, Lei},
  journal={Proceedings of the American Mathematical Society},
  volume={144},
  number={12},
  pages={5291--5308},
  year={2016}
}

@article{LYangProduct,
  title={Equidistribution of expanding translates of curves in homogeneous spaces with the action of {$(\mathrm{SO}(n,1))^k$}},
  author={Yang, Lei},
  journal={Acta Mathematica Sinica, English Series},
  volume={38},
  number={1},
  pages={205--224},
  year={2022},
  publisher={Springer},
  doi={10.1007/s10114-022-0459-1}
}

@book{MilneAG,
  author={Milne, James S.},
  title={Algebraic Groups: The Theory of Group Schemes of Finite Type over a Field},
  series={Cambridge Studies in Advanced Mathematics},
  volume={170},
  publisher={Cambridge University Press},
  address={Cambridge},
  year={2017},
  doi={10.1017/9781316711736}
}

\end{document}